\documentclass[12pt]{amsart}

\usepackage[abs]{overpic}
\usepackage{amssymb}
\usepackage{amsmath}
\usepackage{amsthm}
\usepackage{mathrsfs}
\usepackage{fancybox}
\usepackage{comment}
\usepackage{layout}
\usepackage{color}
\usepackage{bm}

\theoremstyle{plain}
\newtheorem{theorem}{Theorem}[section]
\newtheorem{lemma}[theorem]{Lemma}

\newtheorem{proposition}[theorem]{Proposition}

\theoremstyle{definition}
\newtheorem{definition}[theorem]{Definition}
\newtheorem{example}[theorem]{Example}
\newtheorem{remark}[theorem]{Remark}

\graphicspath{{figures(pdf)/}}

\title[Link-homotopy classes of 4-component links and claspers]{Homotopy classes of 4-component links and claspers}
\author[Kotorii]{Yuka Kotorii}
\address[Kotorii]{Mathematical Analysis Team, RIKEN Center for Advanced Intelligence Project (AIP), 1-4-1 Nihonbashi Chuo-ku Tokyo 103-0027 Japan}
\address{Department of Mathematics, Graduate School of Science, Osaka University, 1-1 Machikaneyama Toyonaka Osaka 560-0043 Japan}
\address{interdisciplinary Theoretical \& Mathematical Sciences Program (iTHEMS) RIKEN, 2-1 Hirosawa Wako Saitama 351-0198 Japan}
\email[Yuka Kotorii]{yuka.kotorii@riken.jp}
\author[Mizusawa]{Atsuhiko Mizusawa}
\address[Mizusawa]{Department of Mathematics, Fundamental Science and Engineering, Waseda University, 3-4-1Okubo, Shinjuku-ku, Tokyo 169-8555, Japan} 
\email[Atsuhiko Mizusawa]{a\_mizusawa@aoni.waseda.jp}
\date{\today}

\begin{document}

\begin{abstract}
Two links are link-homotopic if they are transformed into each other by a sequence of self-crossing changes and ambient isotopies. The link-homotopy classes of 4-component links were classified by Levine with enormous algebraic computations. We modify the results by using Habiro's clasper theory. The new classification gives more symmetrical and schematic points of view to the link-homotopy classes of 4-component links. As applications, we give several new subsets of the link-homotopy classes of 4-component links which are classified by comparable invariants and give an algorithm which determines whether given two links are link-homotopic or not. 
\end{abstract}


\maketitle
\section{Introduction} \label{int}
\par
In this paper, we work in the piecewise linear category and links are ordered and oriented. 
Two links are \textit{link-homotopic} if one is transformed into the other by a sequence of self-crossing changes (i.e. crossing changes between arcs belonging to the same component) and ambient isotopies. The notion of link-homotopy was introduced by Milnor in \cite{Mil01}. In the paper, he defined link-homotopy invariants and gave complete classifications of the link-homotopy classes of 2- and 3-component links. In \cite{Mil02}, Milnor defined finer link-homotopy invariants, which we call \textit{Milnor homotopy invariants}. The Milnor homotopy invariants are equivalent to the invariants in \cite{Mil01} up to 4-component links. (It was shown in \cite{Wa} that there is a pair of 5-component links which are distinguished by the Milnor homotopy invariants but are not by the invariants in \cite{Mil01}.) So link-homotopy classes of 2- and 3-component links are classified by the Milnor homotopy invariants. 
\par
Let $L$ be an $n$-component link.
For a non-repeat sequence $I$ of integers in the set of the component numbers, the Milnor homotopy invariant $\overline{\mu}_L(I)$ for $L$ is defined in the cyclic group $\mathbb{Z}_{\Delta_L (I)}$, 
where $\Delta_L(I)$ is the greatest common divisor of all Milnor homotopy invariants $\overline{\mu}_L(J)$ such that $J$ is obtained from $I$ by removing at least one integer and permuting the remaining ones cyclicly. 
See \cite{Mil02} for details. Note that for the sequence $I=ij$ with length 2, the Milnor homotopy invariant $\overline{\mu}_L(ij)$ is equivalent to the linking number between the $i$-th and $j$-th components. 
The link-homotopy classes of 2-component links $L=K_1\cup K_2$ are classified by the Milnor homotopy invariants $\overline{\mu}_L(12)$ with length 2. The link-homotopy classes of 3-component links $L=K_1\cup K_2 \cup K_3$ are classified by 
the 4 Milnor homotopy invariants $\overline{\mu}_L(12)$, $\overline{\mu}_L(23)$, $\overline{\mu}_L(13)$ and $\overline{\mu}_L(123)$ with length up to 3.\\
\par
Let $\gcd^{\ast} (x_1,\dots,x_n)$ (resp. $\gcd_{\ast} (x_1,\dots,x_n)$) be the greatest common divisor $\gcd (x_1,\dots,$ $ x_n)$ of the integers $x_1, \dots, x_n$ if the integers are not all zeros and $\infty$ (resp. 0) otherwise.
For 4-component links, the link-homotopy classes were classified by Levine \cite{L} as in Theorem \ref{Levine}. 
Let $\mathcal{L}_4$ be the set of link-homotopy classes of 4-component links.

\begin{theorem}[\cite{L}] \label{Levine}
Let $M$ be a set of 12-tuples of integers: $k$, $l$, $r$, $d$ $(0\leq d < \gcd^{\ast} (k,l,r))$ and $e_i$ $(1\leq i \leq 8)$, modulo the relations in Table \ref{L-table}, where the integers $a$, $b$ and $c$ run over all integers satisfying $ak-br+cl=0$. In the table, the relations $\Phi_{j}$ add the entries to $e_i$ $(4 \leq i\leq 8)$. The other integers do not change under the relations. Then there is a bijection between $\mathcal{L}_4$ and $M$.
\end{theorem}

\begin{table}[htb] 
\begin{center}
   \caption{Levine's table of relations.} \label{L-table}
  \begin{tabular}{|c|c|c|c|c|c|} \hline
      & $e_4$ & $e_5$ & $e_6$ & $e_7$ & $e_8$ \\ \hline 
    $\Phi_{1}$ & $k$ & $r$ & 0 & $d$ & $-d$ \\ \hline 
    $\Phi_{2}$ & $-k$ & 0 & $l$ & $-d$ & $0$ \\ \hline 
    $\Phi_{3}$ & $-e_1$ & 0 & $e_3$ & $e_5$ & $0$ \\ \hline 
    $\Phi_{4}$ & $e_2$ & $e_3$ & $0$ & $e_6$ & $-e_6$ \\ \hline 
    $\Phi_{5}$ & $0$ & $-e_1$ & $-e_2$ & $0$ & $e_4$ \\ \hline 
    $\Phi_{6}(a,b,c)$ & $ce_2-be_1$ & $-ae_1$ & $0$ & $ae_4$ & $-abe_1-ce_6+be_5$ \\ \hline 
  \end{tabular} 
\end{center}
\end{table}

\begin{remark} \label{correction-remark}
In Table \ref{L-table}, we correct the sign of the entry $(\Phi_1, e_8)$ in \cite{L}.
\end{remark}

\begin{remark}[\cite{L}] \label{intro-remark}
In Theorem \ref{Levine}, the integers $k$, $l$, $r$, $d$ ($0\leq d < \gcd^{\ast} (k,l,r)$) and $e_i$ $(1\leq i \leq 8)$ have relations to the Milnor homotopy invariants $\overline{\mu}_L$ as follows. 
$$\overline{\mu}_L(12)= k,\, \overline{\mu}_L(23)= l,\, \overline{\mu}_L(13)= r,\, \overline{\mu}_L(14)= e_1,\, \overline{\mu}_L(24)= e_2,\, \overline{\mu}_L(34)= e_3,$$ 
$$\overline{\mu}_L(123)= d,\, \overline{\mu}_L(124)\equiv e_4,\, \overline{\mu}_L(134)\equiv e_5,\, \overline{\mu}_L(234)\equiv e_6,\, \overline{\mu}_L(3124)\equiv e_7,\, \overline{\mu}_L(2134)\equiv e_8,$$
where the integers in the first line are the linking numbers, the first identity in the second line holds by regarding $\overline{\mu}_L(123)$ as an integer in $[0, \gcd^{\ast} (k,l,r) )$, and the other equations in the second line hold in $\mathbb{Z}_{\Delta_L (I)}$ with corresponding sequence $I$.
Here we correct the relations of $e_7$ and $e_8$ in \cite{L} by exchanging them. 
\par
Note that the link-homotopy class of the first 3-component sublink $L_3$ of $L$ is determined by $k$, $l$, $r$ and $d$ (i.e. the Milnor homotopy invariants) and the integers $e_i$ indicate how the fourth component tangles to $L_3$. 
\end{remark}

\par
In \cite{HL}, the link-homotopy classes of string links are classified by using the Milnor homotopy invariants for string links. Furthermore, the link-homotopy classes of links with any number of components are classified as the link-homotopy classes of string links modulo several relations. It was also shown that there is an algorithm which determines whether given two links are link-homotopic or not. However, the link-homotopy classes of links whose numbers of components are more than three are not classified by using comparable invariants.  
\par
For the 4-component link case, Levine  gave several subsets of $\mathcal{L}_4$ which are classified by comparable invariants \cite{L}. We recall the result in Theorem \ref{Levintheorem} below. 
\\

\par
In several cases, the link-homotopy classes are described by using the \textit{claspers} defined in \cite{Ha}. For up to 3-component links, the link-homotopy classes are described by Hopf chords and Borromean chords (i.e. $C_1$- and $C_2$-trees) in \cite{TY}, see also Remarks \ref{3-comp-case} and \ref{3-comp-psi} below. For the link-homotopy classes of string links, the correspondence between the Milnor homotopy invariants and claspers are shown in \cite{Ya}. 
\par
We mention that Nikkuni and the authors also used the clasper theory to study \textit{HL-homotopy classes} for handlebody-links \cite{MN, KM}. \\

\par
In the present paper, we apply the clasper theory to 4-component links and modify Levine's classification result in Theorem \ref{Levine} by introducing new relations as follows.

\begin{theorem} \label{mainthm}
Let $N$ be a set of 12-tuples of integers: $c_1$,\dots, $c_6$, $f_1$, \dots, $f_4$, $t_1$ and $t_2$, modulo the relations $\psi_{ij}$ in Table \ref{originaltable}. Then there is a bijection between $\mathcal{L}_4$ and $N$.
\end{theorem}

\begin{table}[htb] 
\begin{center}
   \caption{The table of relations in Theorem \ref{mainthm}.} \label{originaltable}
  \begin{tabular}{|c|c|c|c|c|c|c|} \hline
      & $f_1$ & $f_2$ & $f_3$ & $f_4$ & $t_1$ & $t_2$ \\ \hline 
    $\psi_{21}$ & 0 & 0 & $c_5$ & $-c_1$ & $f_1$ & $0$ \\ \hline 
    $\psi_{41}$ & 0 & $c_6$ & $-c_5$ & $0$ & $0$ & $-f_1$ \\ \hline 
    $\psi_{12}$ & 0 & 0 & $-c_4$ & $c_2$ & $f_2$ & $0$ \\ \hline 
    $\psi_{32}$ & $c_6$ & 0 & $0$ & $-c_2$ & $0$ & $-f_2$ \\ \hline 
    $\psi_{43}$ & $c_5$ & $-c_4$ & $0$ & $0$ & $f_3$ & $0$ \\ \hline 
    $\psi_{23}$ & $-c_5$ & 0 & $0$ & $c_3$ & $0$ & $-f_3$ \\ \hline 
    $\psi_{34}$ & $-c_1$ & $c_2$ & $0$ & $0$ & $f_4$ & $0$ \\ \hline 
    $\psi_{14}$ & 0 & $-c_2$ & $c_3$ & $0$ & $0$ & $-f_4$ \\ \hline 
  \end{tabular} \\
\end{center}
\end{table}

\begin{remark} \label{intro-remark2}
The two tables in Theorem \ref{Levine} and \ref{mainthm} are related as follows.
The relations between the integers $k, l, r, d, e_1, \dots, e_7$ and $e_8$ in Theorem \ref{Levine} and the integers $c_1, \dots, c_6, f_1, \dots, f_4$, $t_1$ and $t_2$ in Theorem \ref{mainthm} are 
$$k=-c_3,\, l=-c_1,\, r=-c_2,\, e_1=-c_4,\, e_2=-c_5,\, e_3=-c_6,$$ 
$$d=f_4,\, e_4=-f_3,\, e_5=f_2,\, e_6=-f_1,\, e_7=-t_2\ \mbox{ and } e_8=t_1+t_2.$$
Note that if $f_4$ does not satisfy $0\leq f_4< \gcd^{\ast} (c_1, c_2, c_3)$, we need to make $f_4$ in the range by using relations $\psi_{ij}$ in Table \ref{originaltable} (accordingly $f_1, \dots, f_3, t_1, t_2$ also change), to see the relations. 
The relations $\Phi_i$ are described by using relations $\psi_{ij}$: 
$$\Phi_1=\psi_{14},\, \Phi_2=(\psi_{14}\psi_{34})^{-1},\, \Phi_3=\psi_{12}\psi_{32},\,\Phi_4=\psi_{41}^{-1},$$
$$\Phi_5=\psi_{43}^{-1}\,\mbox{ and }\Phi_6(a,b,c)=\psi_{21}^c\psi_{12}^b\{(\psi_{23}\psi_{43})^{-1}\}^a.$$
\end{remark}

\par
The proof of Theorem \ref{mainthm} is done by using clasper presentations of 4-component link-homotopy classes, see Section \ref{main}. These clasper presentations give more symmetrical and schematic points of view to the link-homotopy classes of 4-component links. We also give new subsets of $\mathcal{L}_4$ which are classified by comparable invariants other than Levine's ones in \cite{L}. Moreover, we also give an algorithm which determines whether given two 4-component links are link-homotopic or not by translating the algorithm in \cite{HL} into our notation. Examples are given. 


\par
The present paper is organized as follows. In Section \ref{clasp}, we review the clasper theory quickly and give a standard form of 4-component links up to link-homotopy by using the claspers. In Section \ref{main}, we prove Theorem \ref{mainthm}. We give new subsets of $\mathcal{L}_4$ which are classified by comparable invariants and also give the algorithm which determines whether given two 4-component links are link-homotopic or not in Section \ref{classif}. 

\section*{Acknowledgement}
The authors thank Professor Jean-Baptiste Meilhan for valuable comments and suggestions. The first author is supported by RIKEN AIP (Center for Advanced Intelligence Project) and RIKEN iTHEMS Program. The first author is partially supported by Grant-in-Aid for Young Scientists (B) (No. 16K17586), Japan Society for the Promotion of Science.

\section{Claspers and standard forms} \label{clasp}
\par
In this section, we define a $C_k$-tree which is a special version of a clasper and transform 4-component links into standard forms up to link-homotopy. The standard forms consist of trivial links and $C_1$-, $C_2$- and $C_3$-trees. The standard form is not unique, but it will be shown that the form is unique up to the relations $\psi_{ij}$ in Table \ref{originaltable}. 

\begin{definition}\label{clasper}(\cite{Ha})
A disk $T$ embedded in $S^3$ is called a \textit{simple tree clasper} 
for a link $L$ if it satisfies the following three conditions:
\begin{enumerate} 
\item The embedded disk $T$ is decomposed into bands and disks, where each band connects two distinct disks and each disk attaches either 1 or 3 bands.
We call a band an \textit{edge} and a disk attached 1 band a \textit{leaf}. 
\item The embedded disk $T$ intersects the link $L$ transversely so that the intersections are contained in the interiors of the leaves. 
\item Each leaf intersects $L$ at exactly one point.
\end{enumerate}
In this paper, we call a simple tree clasper with $k+1$ leaves a \textit{$C_k$-tree}. In figures, we express disks and bands as follows,
$$\raisebox{-20 pt}{\begin{overpic}[width=40
pt]{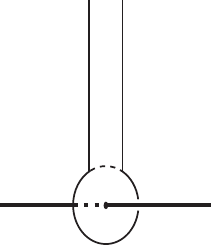}
\put(-3,11){\footnotesize leaf}
\end{overpic}}
  \hspace{0.2cm} \mbox{\large$\longrightarrow$}
\hspace{0.2cm}
\raisebox{-20 pt}{\begin{overpic}[width=40
pt]{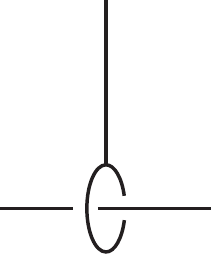}
\end{overpic}}\,,
\hspace{1.0cm}
\raisebox{-15 pt}{\begin{overpic}[width=50
pt]{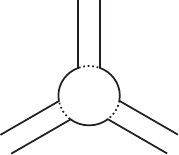}
\end{overpic}}
  \hspace{0.2cm} \mbox{\large$\longrightarrow$}
\hspace{0.2cm}
\raisebox{-15 pt}{\begin{overpic}[width=50
pt]{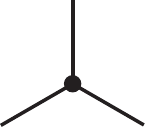}
\end{overpic}}\,.
$$

\par
Given a $C_k$-tree $T$ for a link $L$, there exists a procedure to construct a
framed link in a regular neighborhood $N(T)$ of $T$. 
We call a surgery along the framed link a \textit{surgery along} $T$. 
The surgery gives an orientation-preserving homeomorphism which fixes the boundary, 
from $N(T)$ to the manifold obtained from $N(T)$ by the surgery along $T$. 
We can regard the surgery along $T$ as a local move on $L$, which we call \textit{$C_k$-move} if $T$ is a $C_k$-tree. A $C_k$-move is regarded as a band sum of a Brunnian link.
An example of a $C_k$-move is showed in Figure~\ref{Ck-tree}. 
We denote by $L_T$ the link obtained from $L$ by surgery along $T$. For a family $\mathbb{T}$ of simple tree claspers for $L$, we identify $L\cup \mathbb{T}$ with $L_{\mathbb{T}}$. 

\begin{figure}[ht]
$$\raisebox{-15 pt}{\begin{overpic}[bb=0 0 217 73, width=160
pt]{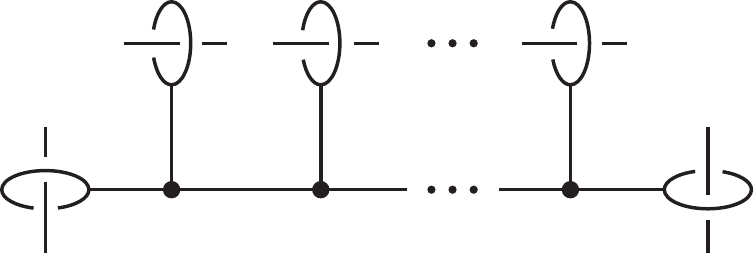}
\end{overpic}}
  \hspace{0.5cm} \mbox{\large$\xrightarrow[]{\rm surgery}$}
\hspace{0.5cm}
\raisebox{-33 pt}{\begin{overpic}[bb=0 0 229 9, width=180
pt]{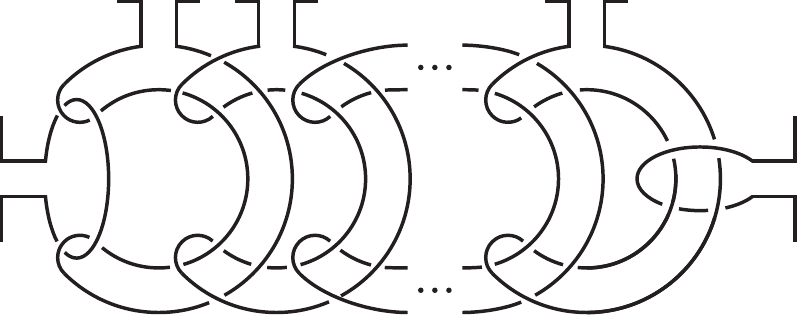}
\end{overpic}}$$
    \caption{An example of a $C_k$-move.}
    \label{Ck-tree}
\end{figure}
\end{definition}

We remark that the surgeries along $C_1$- and $C_2$-trees are taking a band sum of a Hopf link and a Borromean ring respectively, see Figure \ref{C1,C2-tree}. 

\begin{figure}[ht]
$$\raisebox{-43 pt}{\begin{overpic}[width=40
pt]{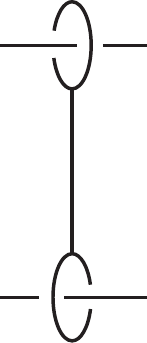}
\end{overpic}}
  \hspace{0.5cm} \mbox{\large$\xrightarrow[]{\rm surgery}$}
\hspace{0.5cm}
\raisebox{-30 pt}{\begin{overpic}[width=40
pt]{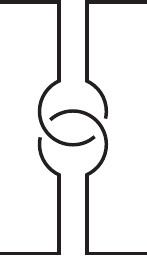}
\end{overpic}}
\hspace{1.5cm}
\raisebox{-24 pt}{\begin{overpic}[width=70
pt]{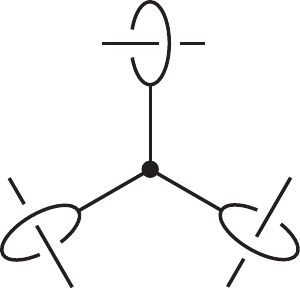}
\end{overpic}}
  \hspace{0.5cm} \mbox{\large$\xrightarrow[]{\rm surgery}$}
\hspace{0.5cm}
\raisebox{-27 pt}{\begin{overpic}[width=70
pt]{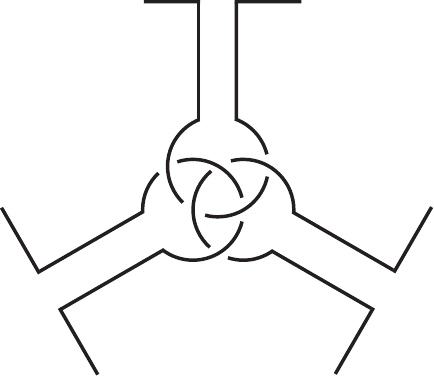}
\end{overpic}}
$$
    \caption{Surgeries along $C_1$- and $C_2$-trees.}
    \label{C1,C2-tree}
\end{figure}

In the diagrams, we express a half-twist of a band as in Figure \ref{half-twist}. 

\begin{figure}[ht]
$$\raisebox{-21 pt}{\begin{overpic}[height=50
pt]{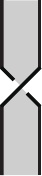}
\end{overpic}}
  \hspace{0.2cm} \mbox{\large$\longrightarrow$}
\hspace{0.2cm}
\raisebox{-21 pt}{\begin{overpic}[height=50
pt]{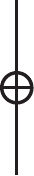}
\end{overpic}}\,
\hspace{1.5cm}
\raisebox{-21 pt}{\begin{overpic}[height=50
pt]{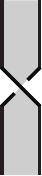}
\end{overpic}}
  \hspace{0.2cm} \mbox{\large$\longrightarrow$}
\hspace{0.2cm}
\raisebox{-21 pt}{\begin{overpic}[height=50
pt]{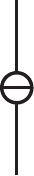}
\end{overpic}}
$$
    \caption{Half-twists on an edge.}
    \label{half-twist}
\end{figure}

We prepare lemmas of moves of simple tree claspers to transform shapes of links.

\begin{lemma}[\cite{Ha}] \label{crossing-change}
The following relation holds up to ambient isotopy. Thus a crossing change is achieved by attaching a $C_1$-tree. 
$$
\raisebox{-20 pt}{\begin{overpic}[width=150pt]{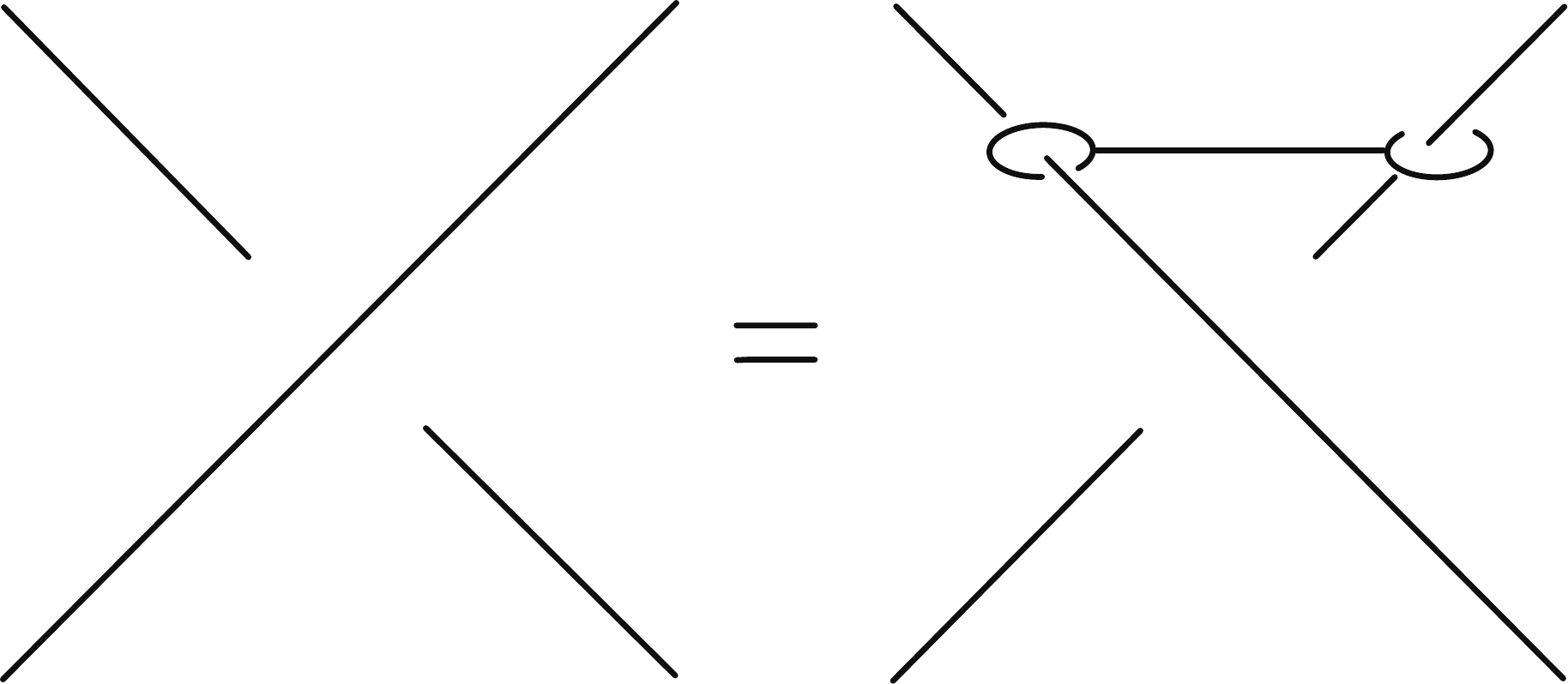}
\end{overpic}}
$$ 
\end{lemma}

Since a surgery along $C_k$-trees corresponds to a band sum of a Brunnian link, the following two lemmas hold.

\begin{lemma} \label{vanish}
If two leaves of a $C_k$-tree intersect the same component of a link, the $C_k$-tree vanishes up to link-homotopy. 
\end{lemma}

\begin{lemma} \label{edge-self-change}
A crossing change between edges belonging to a $C_k$-tree is achieved by link-homotopy.
\end{lemma}
 
\par
There is a lemma to cancel two claspers. This lemma is originally stated up to $C_k$-moves for an integer $k$ in \cite{Ha}. We rephrase it up to link-homotopy. The proof is obtained from the original proposition by using Lemma \ref{vanish}.

\begin{lemma}[\cite{Ha}] \label{para-cancel}
For two copies of $C_k$-trees, if we add an extra (positive or negative) half-twist to the one of them then they cancel up to link-homotopy. 
$$
\raisebox{-20 pt}{\begin{overpic}[width=130pt]{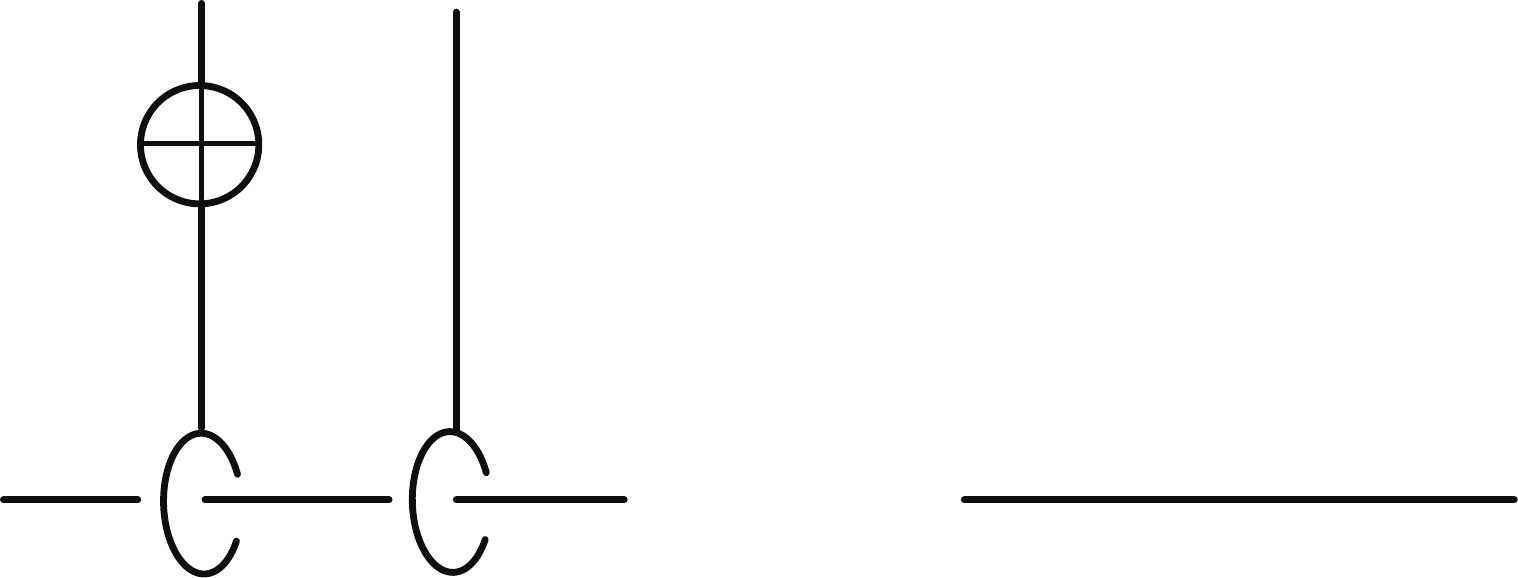}
\put(61,23){\mbox{$\overset{\mbox{\rm l.h.}}{\sim}$}}
\end{overpic}}
$$ 
\end{lemma}

\par
The following three lemmas are proved by using Lemma \ref{para-cancel}.

\begin{lemma} \label{fulltwist-vanish}
A full-twist of an edge of a $C_k$-tree vanishes up to link-homotopy.
$$
\raisebox{-20 pt}{\begin{overpic}[width=100pt]{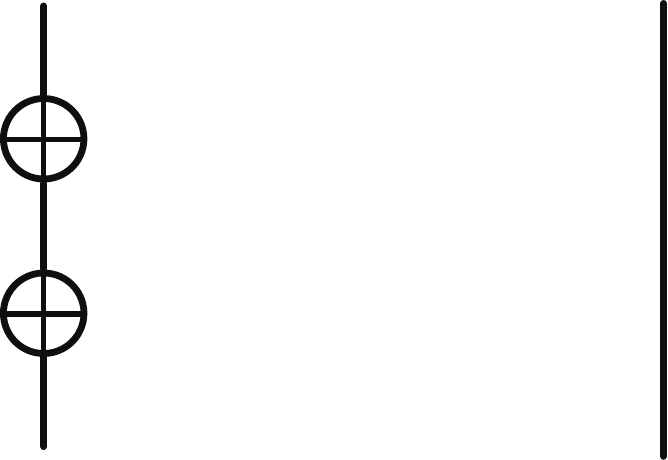}
\put(43,30){\mbox{$\overset{\mbox{\rm l.h.}}{\sim}$}}
\end{overpic}}
$$ 
\end{lemma}

\begin{lemma} \label{half-twist-move}
For a vertex of a $C_k$-tree, a half-twist of an incident edge can be moved to another incident edge up to link-homotopy.
$$
\raisebox{-20 pt}{\begin{overpic}[width=150pt]{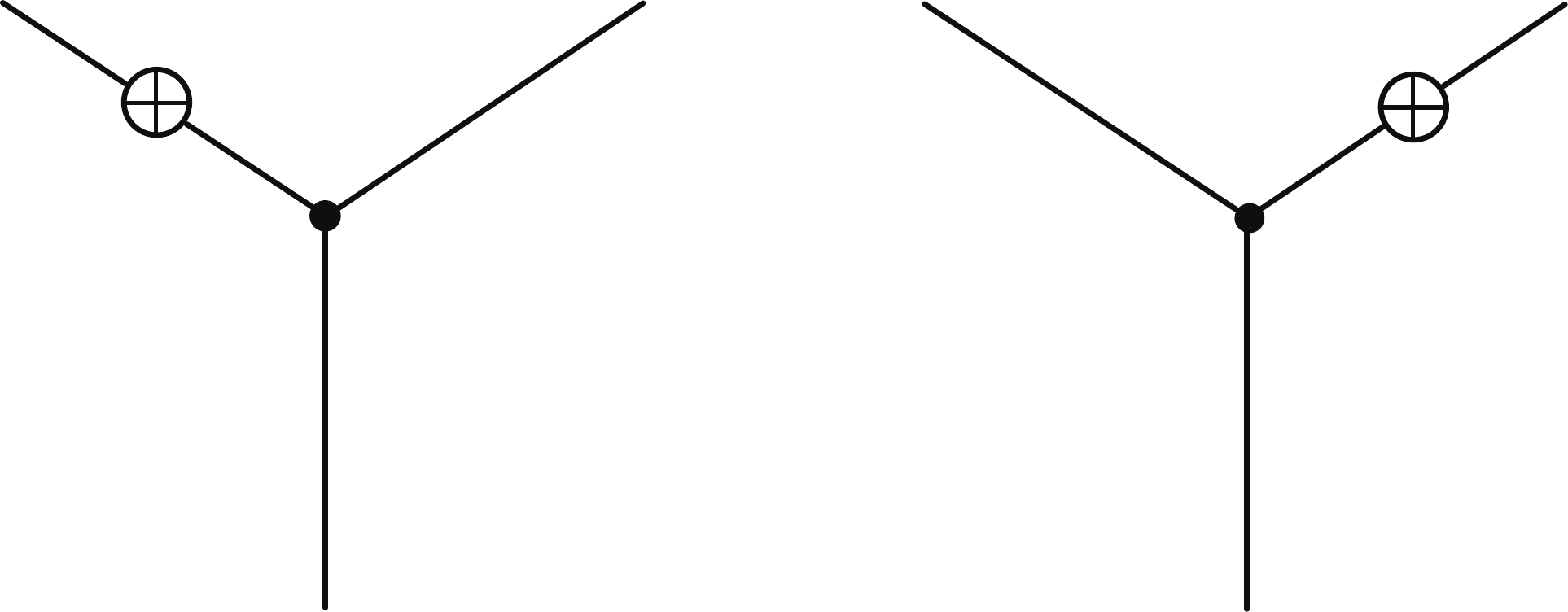}
\put(69,26){\mbox{$\overset{\mbox{\rm l.h.}}{\sim}$}}
\end{overpic}}
$$ 
\end{lemma}

\begin{lemma} \label{vertex-twist}
A twist of two incident edges of a trivalent vertex of a $C_k$-tree can be changed to a half-twist of the other incident edge of the vertex up to link-homotopy. 
$$
\raisebox{-20 pt}{\begin{overpic}[width=150pt]{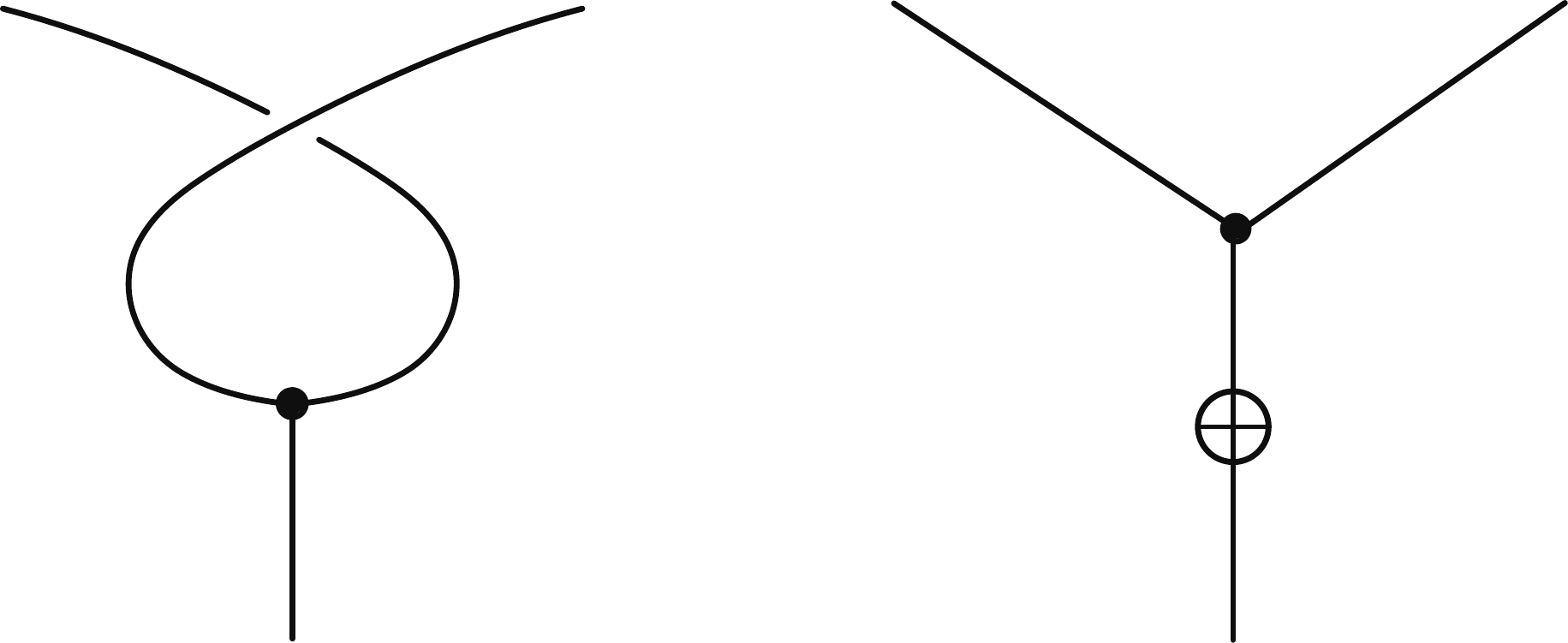}
\put(69,26){\mbox{$\overset{\mbox{\rm l.h.}}{\sim}$}}
\end{overpic}}
$$
\end{lemma}

\par
By Lemma \ref{half-twist-move}, half-twists on a $C_k$-tree are gathered to an edge incident to a leaf and by Lemma \ref{fulltwist-vanish}, all the half-twists vanish if the number of the half-twists is even or a half-twist is left otherwise. We call the former $C_k$-tree \textit{non-twisted} and the latter \textit{twisted}. We assume that the half-twist of twisted $C_k$-tree is at the edge incident to the leaf intersecting the highest order component among the leaves of the $C_k$-tree. 
\par
 The following four lemmas are originally stated up to $C_k$-moves for an integer $k$ in the previous works. In this paper, we rephrase them up to link-homotopy by using Lemma \ref{vanish}.

\begin{lemma}[\cite{Ha, MY}] \label{edge-cross-change}
A crossing change between a $C_i$-tree and a $C_j$-tree makes a new $C_{i+j+1}$-tree which is a union of copies of the two simple tree claspers connected by a new edge as in the figure.
$$
\raisebox{-20 pt}{\begin{overpic}[width=150pt]{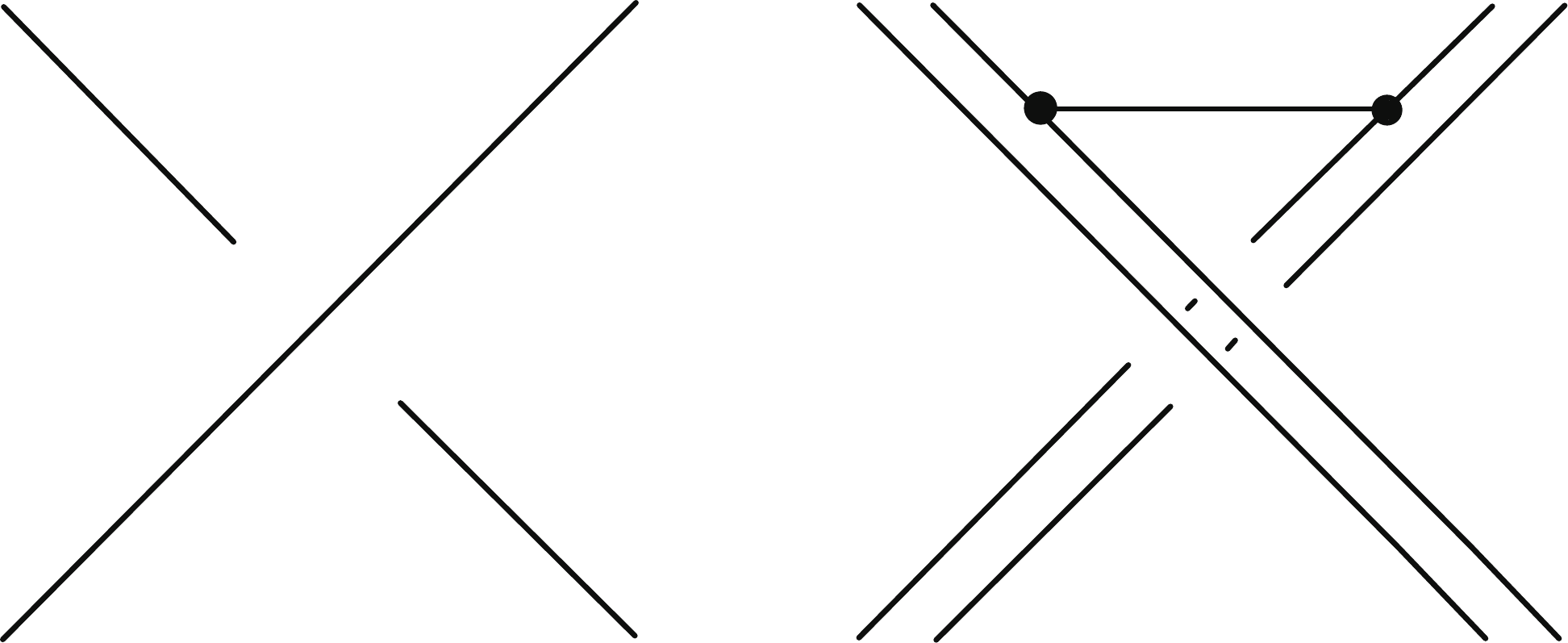}
\put(65,26){\mbox{$\overset{\mbox{\rm l.h.}}{\sim}$}}
\put(4,60){\footnotesize $C_i$}\put(46,60){\footnotesize $C_j$}
\put(75,50){\footnotesize $C_i$}\put(146,50){\footnotesize $C_j$}
\put(101,55){\footnotesize $C_{i+j+1}$}
\end{overpic}}
$$ 
\end{lemma}

Remark that, in Lemma \ref{edge-cross-change}, if the $C_i$- and $C_j$-trees in the left-hand side have leaves at the same component, the new clasper in the right-hand side vanishes by Lemma \ref{vanish}.

\begin{lemma}[\cite{Ha, MY}] \label{leaf-exchange}
An exchange of leaves of a $C_i$-tree and a $C_j$-tree makes a new $C_{i+j}$-tree which is a fusion of copies of the two simple tree claspers as in the figure. 
$$
\raisebox{-20 pt}{\begin{overpic}[width=170pt]{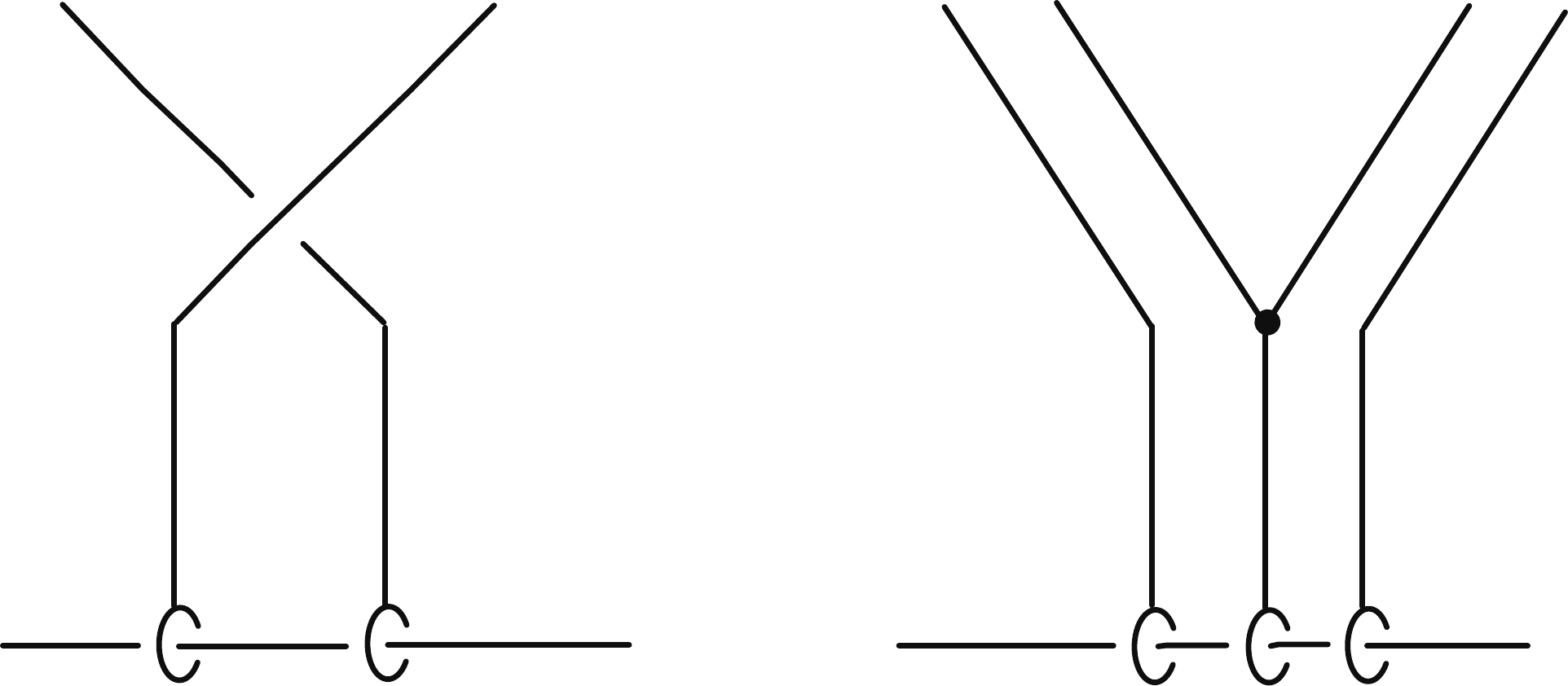}
\put(76,31){\mbox{$\overset{\mbox{\rm l.h.}}{\sim}$}}
\put(-5,65){\footnotesize $C_i$}\put(55,65){\footnotesize $C_j$}
\put(92,65){\footnotesize $C_i$}\put(170,65){\footnotesize $C_j$}
\put(122,65){\footnotesize $C_{i+j}$}
\end{overpic}}
$$ 
\end{lemma}

Remark that, in Lemma \ref{leaf-exchange}, if the $C_i$- and $C_j$-trees in the left-hand side have leaves at the same component other than the ones shown in the figure, the new clasper in the right-hand side vanishes by Lemma \ref{vanish}.

\begin{lemma}[\cite{M}] \label{IHX}
For a $C_k$-tree, the following relation holds up to link-homotopy. 
$$
\raisebox{-20 pt}{\begin{overpic}[width=170pt]{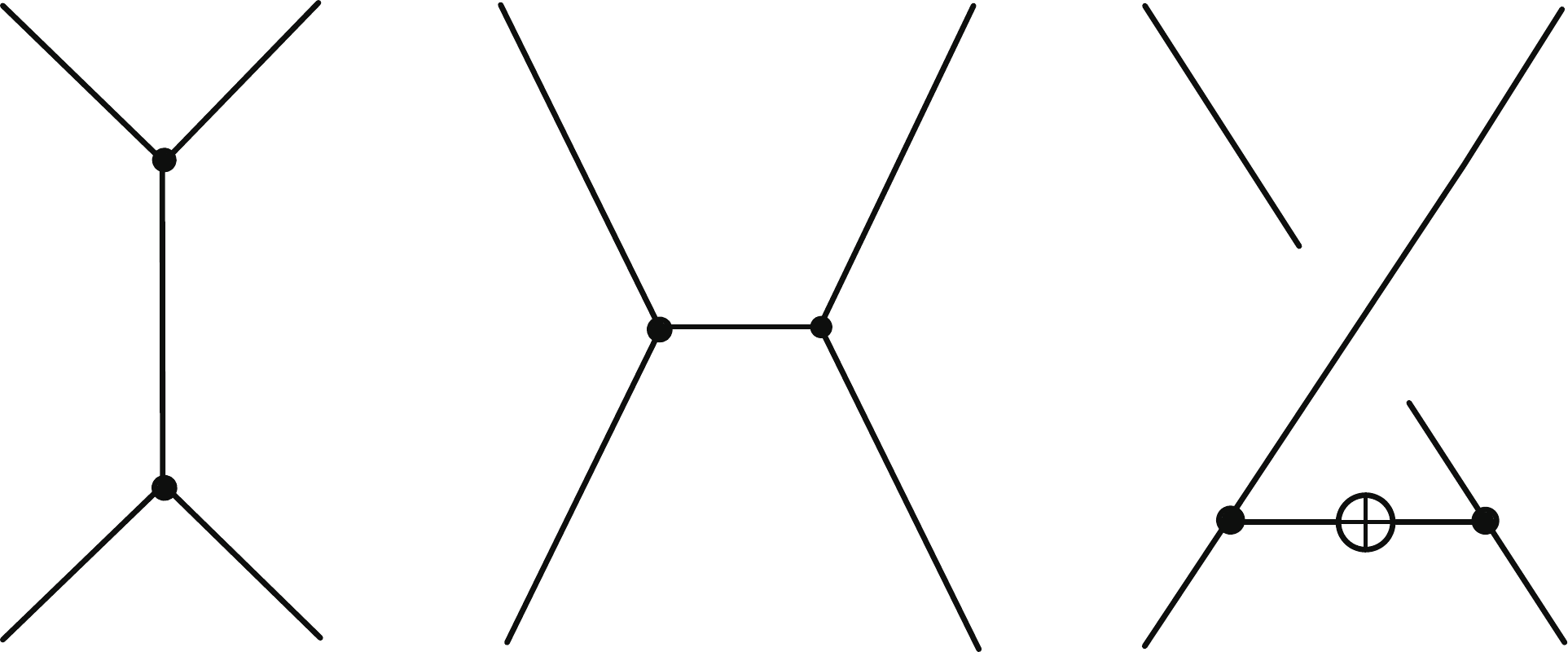}
\put(38,32){\mbox{$\overset{\mbox{\rm l.h.}}{\sim}$}}
\put(110,33){$\cup$}
\end{overpic}}
$$
Here the $C_k$-trees in the right-hand side are two copies of the $C_k$-tree in the left-hand side in its small neighborhood which differ in a small ball as shown respectively. 
This is well-defined up to link-homotopy from the remarks after Lemmas \ref{edge-cross-change} and \ref{leaf-exchange}. 
\end{lemma}

\begin{lemma}[\cite{Ha,FY}] \label{edge-comp-change}
A crossing change between a $C_k$-tree and a link component makes a new $C_{k+1}$-tree up to link-homotopy as follows. 
$$
\raisebox{-23 pt}{\begin{overpic}[width=80pt]{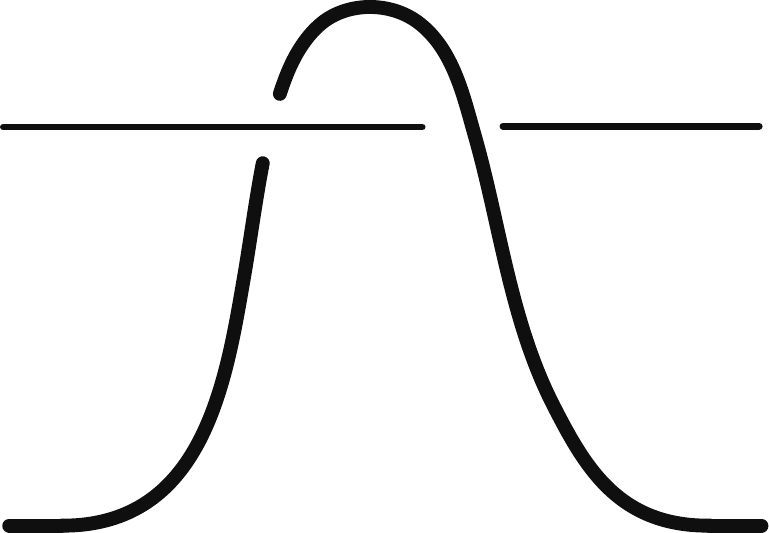}
\put(2,46){\footnotesize $C_k$}
\end{overpic}}
\,\,\,\,
\mbox{$\overset{\mbox{\rm l.h.}}{\sim}$}
\,\,\,\,
\raisebox{-28 pt}{\begin{overpic}[width=80pt]{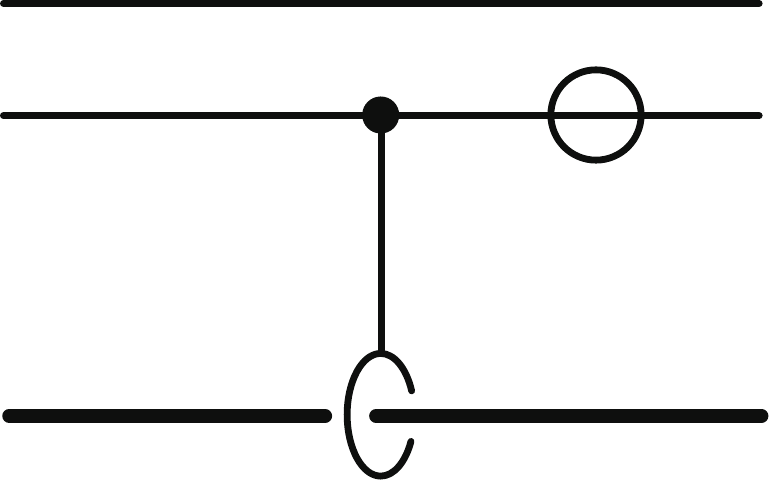}
\put(2,53){\footnotesize $C_k$}\put(2,26){\footnotesize $C_{k+1}$}
\end{overpic}}
$$
Here the bold arcs are link components, the new $C_{k+1}$-tree in the right-hand side is a copy of the $C_k$-tree with a new vertex and a new edge connecting the vertex with the new leaf intersecting with the component and the new $C_{k+1}$-tree has an extra minus half-twist.
\end{lemma}

Remark that, in Lemma \ref{edge-comp-change}, if the $C_k$-tree in the left-hand side has a leaf intersecting the component, the new clasper in the right-hand side vanishes by Lemma \ref{vanish}. 

\begin{example} \label{example-edge-comp-change}
We show the relations in Lemma \ref{edge-comp-change} for $C_1$- and $C_2$-trees explicitly for later use. 
$$
(1)\hspace{0.3cm} \raisebox{-26 pt}{\begin{overpic}[width=150pt]{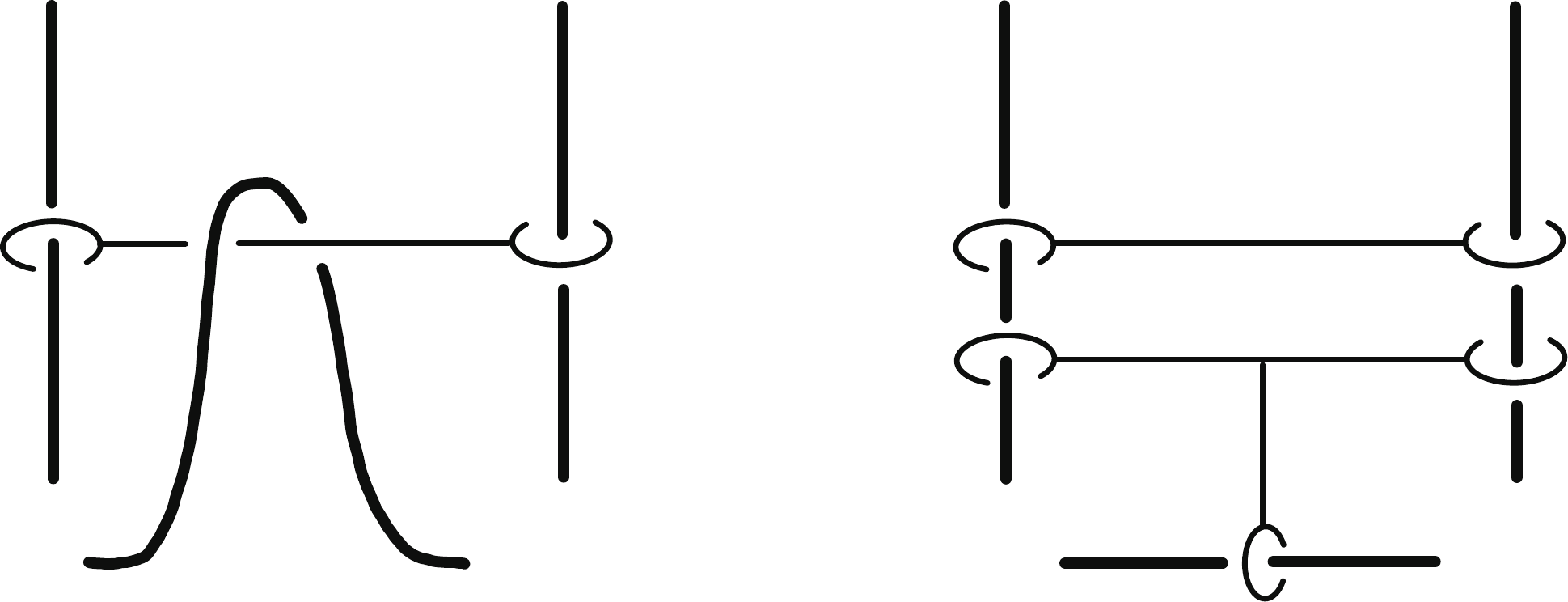}
\put(69,26){$\overset{\mbox{\rm\footnotesize l.h.}}{\mbox{\large$\sim$}}$}
\end{overpic}}
\,,
\hspace{0.6cm}
(2)\hspace{0.3cm}
\raisebox{-26 pt}{\begin{overpic}[width=150pt]{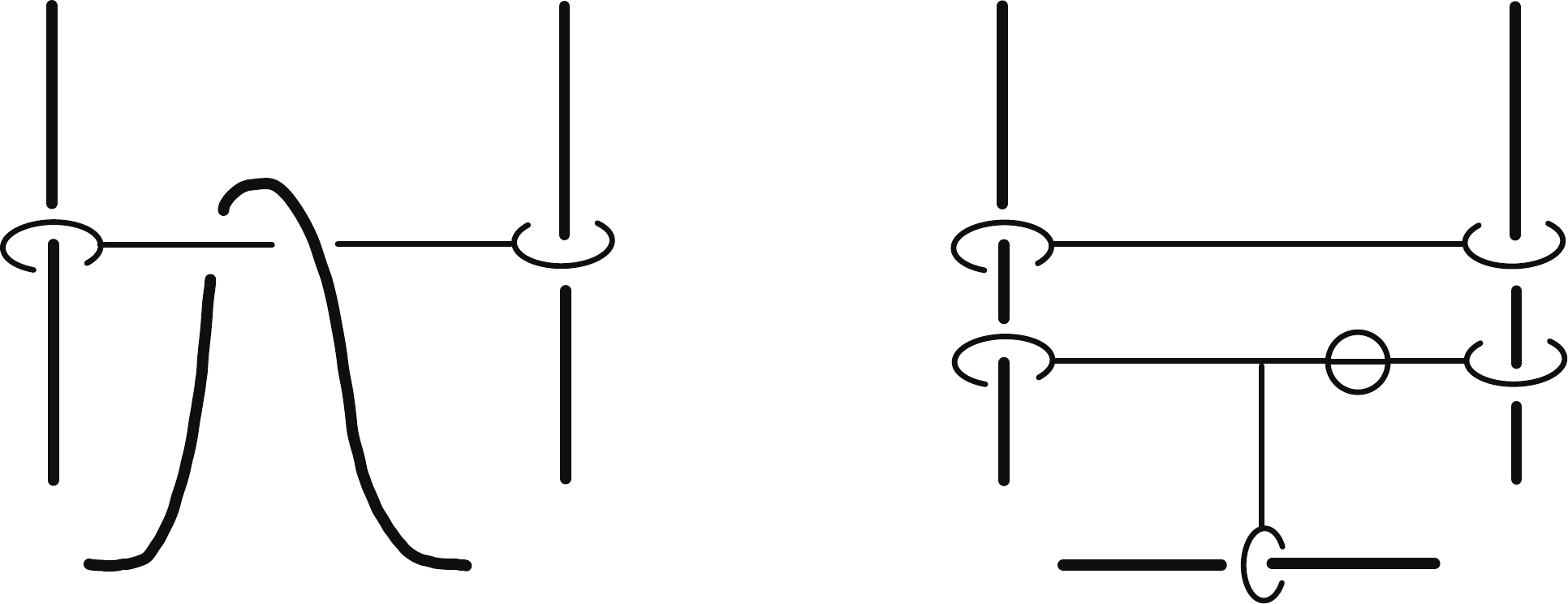}
\put(69,26){$\overset{\mbox{\rm\footnotesize l.h.}}{\mbox{\large$\sim$}}$}
\end{overpic}}\,,
$$

$$
(3)\hspace{0.3cm}\raisebox{-26 pt}{\begin{overpic}[width=150pt]{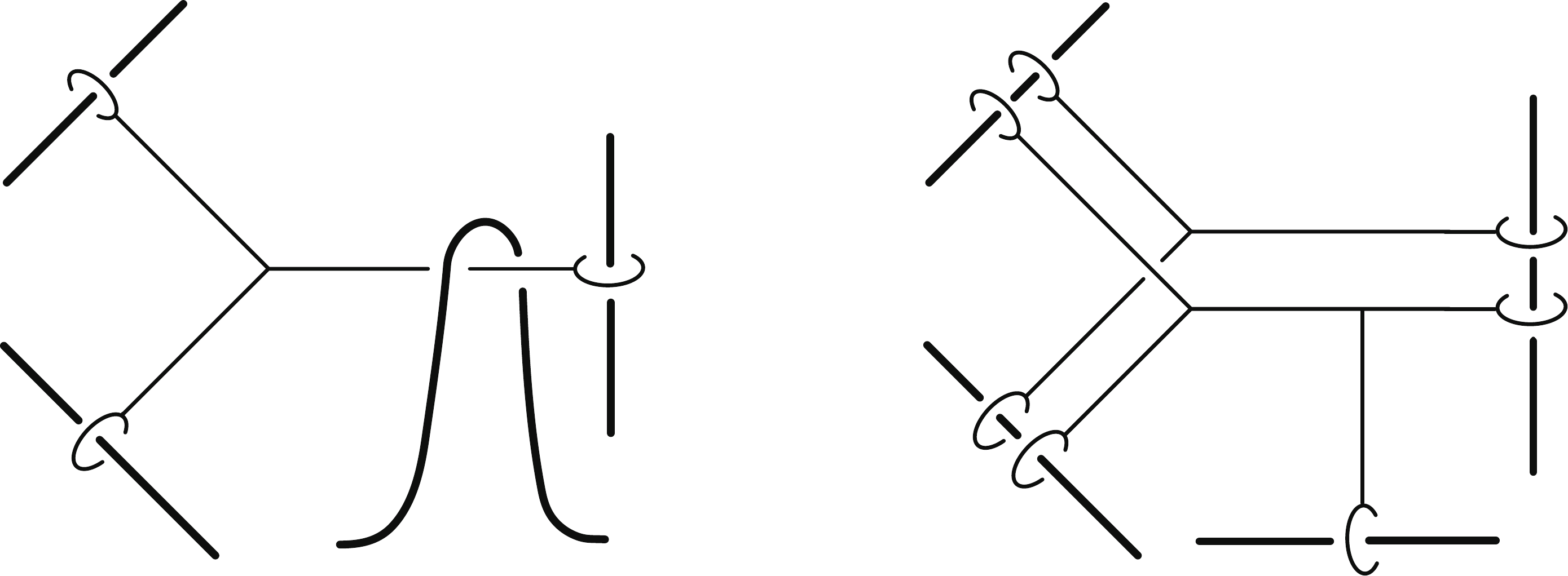}
\put(69,24){$\overset{\mbox{\rm\footnotesize l.h.}}{\mbox{\large$\sim$}}$}
\end{overpic}}
\,,
\hspace{0.6cm}
(4)\hspace{0.3cm}\raisebox{-26 pt}{\begin{overpic}[width=150pt]{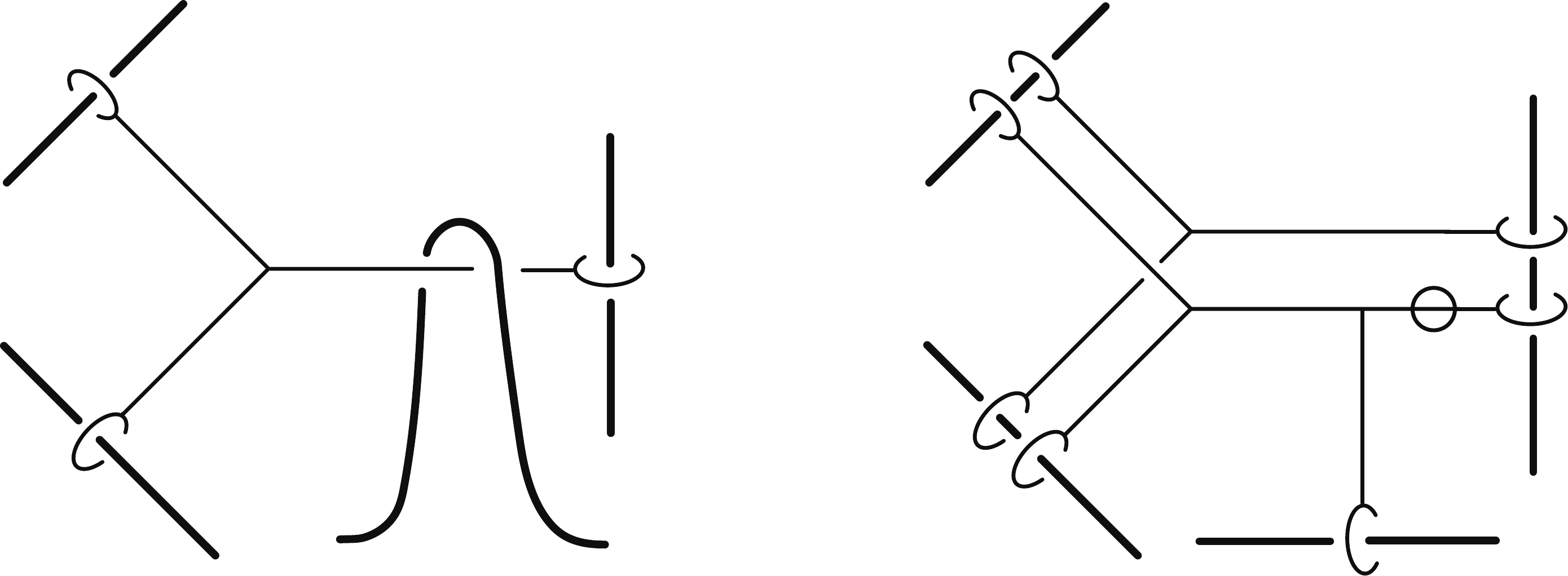}
\put(69,24){$\overset{\mbox{\rm\footnotesize l.h.}}{\mbox{\large$\sim$}}$}
\end{overpic}}\,.
$$
\end{example}

\par
Lemma \ref{edge-self-change} and the remark after Lemma \ref{edge-cross-change} show that some crossing changes between edges of claspers are achieved by link-homotopy. From now on, we sometimes omit over-under information of crossings in figures if the differences vanish up to link-homotopy. 
\par
Let a $C_k$-tree attached an integer $s$ be $s$ parallel copies of non-twisted $C_k$-trees if $0\leq s$, and $|s|$ parallel copies of twisted $C_k$-trees if $s<0$. For $0\leq s$, the $s$ parallel simple tree claspers can be presented as in Figure \ref{parallel clasper} up to link-homotopy. If $s<0$,  each clasper has a half-twist. Note that in this presentation there are ambiguities of the choices of over-under information at the crossings between the edges of the claspers and the arrangement of their leaves. However, since the remarks after Lemmas \ref{edge-cross-change} and \ref{leaf-exchange}, the ambiguities vanish up to link-homotopy. 

\begin{figure}[ht]
$$
\raisebox{-30 pt}{\begin{overpic}[height=70pt]{1-clasper01.pdf}
\put(5,33){$s$}
\end{overpic}}
\,\,\,\,\longrightarrow\,\,\,\,
\raisebox{-40 pt}{\begin{overpic}[height=80pt]{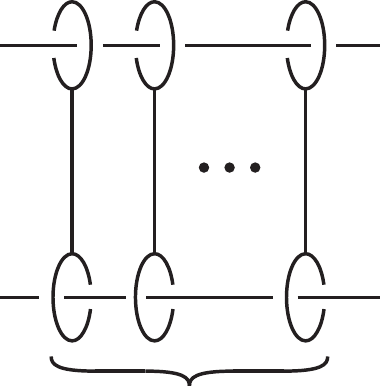}
\put(36,-9){$s$}
\end{overpic}}
\quad\quad
\raisebox{-30 pt}{\begin{overpic}[height=70pt]{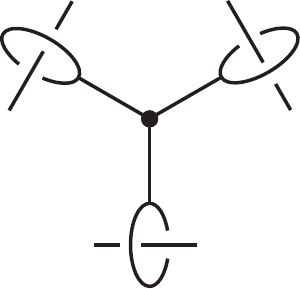}
\put(26,31){$s$}
\end{overpic}}
\,\,\,\,\longrightarrow\,\,\,\,
\raisebox{-35 pt}{\begin{overpic}[height=80pt]{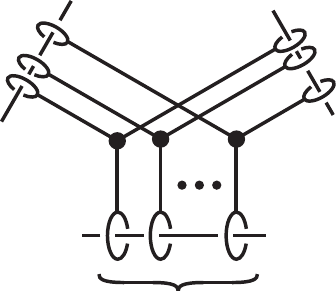}
\put(46,-9){$s$}
\end{overpic}}$$\\

$$
\raisebox{-25 pt}{\begin{overpic}[height=60pt]{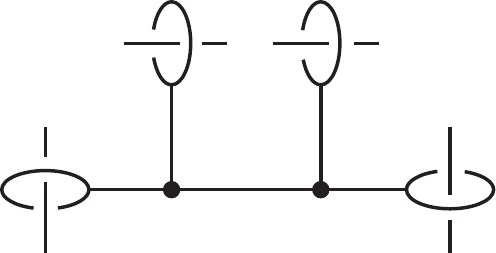}
\put(56,5){$s$}
\end{overpic}}
\,\,\,\,\longrightarrow\,\,\,\,
\raisebox{-25 pt}{\begin{overpic}[height=60pt]{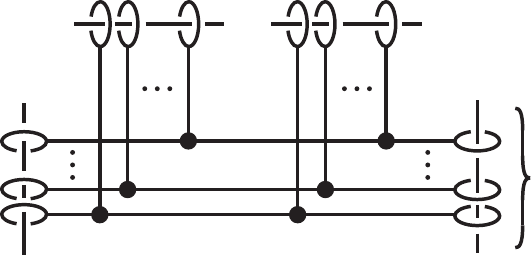}
\put(128,15){$s$}
\end{overpic}}
$$
\caption{Parallel simple tree claspers.} \label{parallel clasper}
\end{figure}

\par
Let $L$ be a 4-component link. Now we transform $L$ into a standard form by using lemmas above as the following steps. 
\par
(1) By Lemma \ref{crossing-change}, $L$ can be transformed into a trivial link with $C_1$-trees attached. 
\par
(2) By using Lemmas \ref{edge-cross-change}, \ref{leaf-exchange} and \ref{edge-comp-change}, for any pair of the components, we transform the $C_1$-trees in parallel and, by using Lemmas \ref{para-cancel} and \ref{fulltwist-vanish}, make them all non-twisted or twisted for each pair of components. There are $C_i$-trees $(i=2, 3)$ attaching to the components which occur in the moves of Lemmas \ref{edge-cross-change}, \ref{leaf-exchange} and \ref{edge-comp-change}. Note that, from Lemma \ref{vanish}, $C_i$-trees vanish if $4\leq i$. 

\par
(3) By using Lemmas \ref{vertex-twist}, \ref{edge-cross-change}, \ref{leaf-exchange} and \ref{edge-comp-change}, we transform the $C_2$-trees in parallel for any 3-tuple of the components and, by using Lemmas \ref{para-cancel}, \ref{fulltwist-vanish} and \ref{half-twist-move}, make them all non-twisted or twisted for each 3-tuple of components. We fix a configuration of the $C_1$- and $C_2$-trees as in Figure \ref{canon-form} left.

\par
(4) There are $C_3$-trees attaching to the components derived from steps (2) and (3). By using Lemmas \ref{vertex-twist}, \ref{edge-cross-change}, \ref{leaf-exchange}, \ref{IHX} and \ref{edge-comp-change}, we transform the $C_3$-trees in parallel so that, if we forget twists of edges, there are only two types of $C_3$-trees depicted in Figure \ref{canon-form} middle and right and, by using Lemmas \ref{para-cancel}, \ref{fulltwist-vanish} and \ref{half-twist-move}, make them all are non-twisted or twisted for each type.
\par
 We call the shape of the trivial link with the $C_1$-, $C_2$- and $C_3$-trees in Figure \ref{canon-form} a \textit{standard form} of $L$, where the alphabets near the claspers are the numbers of the parallel claspers. Here we depict $C_3$-trees separately for simplicity since, from the remarks after Lemmas \ref{edge-cross-change} and \ref{leaf-exchange}, 
the positions of leaves of $C_3$-trees do not affect the numbers of other claspers.

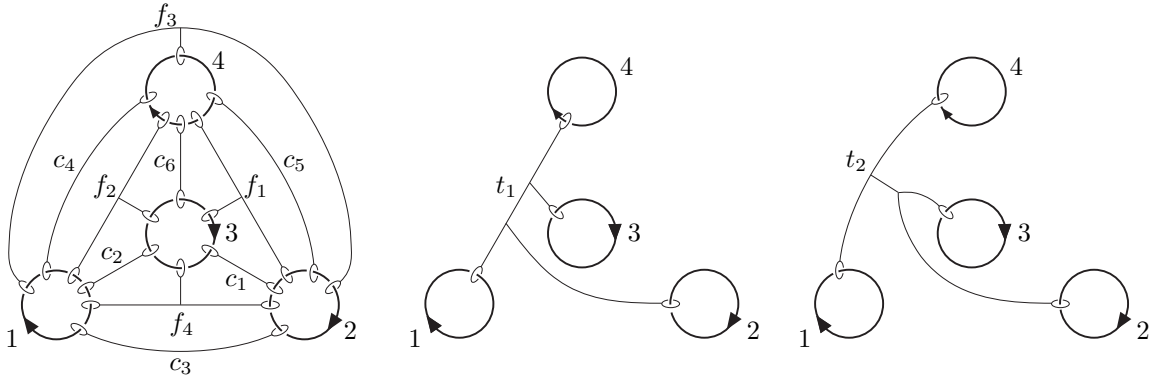
\begin{figure}[ht]
\raisebox{0 pt}{\begin{overpic}[width=130pt]{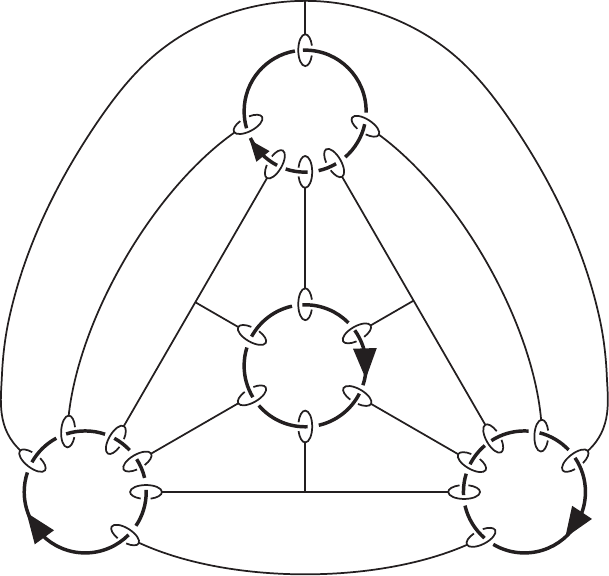}
\put(-1,2){\footnotesize \textbf{$1$}}
\put(127,5){\footnotesize \textbf{$2$}}
\put(82,42){\footnotesize \textbf{$3$}}
\put(77,107){\footnotesize \textbf{$4$}}
\put(82,25){\footnotesize$c_1$}
\put(35,35){\footnotesize$c_2$}
\put(61,-7){\footnotesize$c_3$}
\put(17,70){\footnotesize$c_4$}
\put(104,70){\footnotesize$c_5$}
\put(55,70){\footnotesize$c_6$}
\put(89,60){\footnotesize$f_1$}
\put(32,60){\footnotesize$f_2$}
\put(55,125){\footnotesize$f_3$}
\put(61,8){\footnotesize$f_4$}
\end{overpic}}
\qquad
\raisebox{5 pt}{\begin{overpic}[width=120pt]{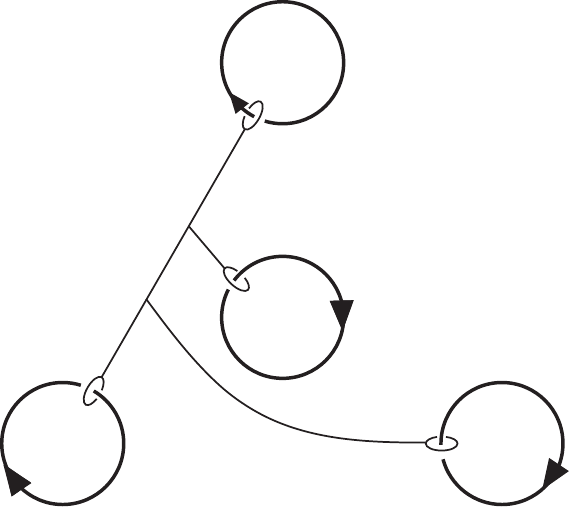}
\put(-6,-3){\footnotesize \textbf{$1$}}
\put(122,0){\footnotesize \textbf{$2$}}
\put(77,37){\footnotesize \textbf{$3$}}
\put(74,100){\footnotesize \textbf{$4$}}
\put(27,55){\footnotesize$t_1$}
\end{overpic}}
\qquad
\raisebox{5 pt}{\begin{overpic}[width=120pt]{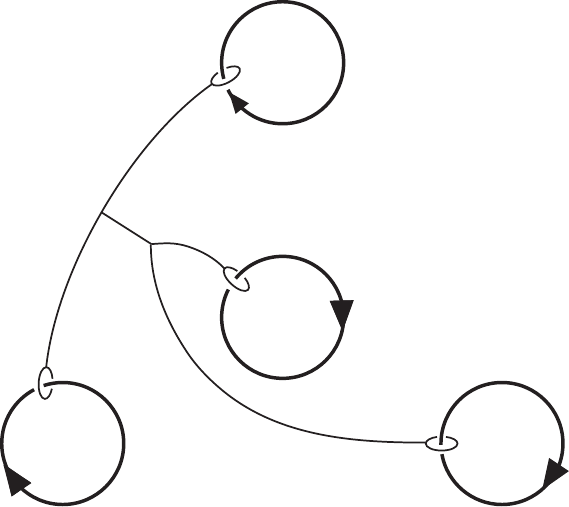}
\put(-6,-3){\footnotesize \textbf{$1$}}
\put(122,0){\footnotesize \textbf{$2$}}
\put(77,37){\footnotesize \textbf{$3$}}
\put(74,100){\footnotesize \textbf{$4$}}
\put(12,65){\footnotesize$t_2$}
\end{overpic}}
\caption{The standard form of $L$.} \label{canon-form}
\end{figure}

\par
Note that the configuration of claspers of the standard form is arranged so that, if we regard the components of the trivial link as vertices of a tetrahedron, the $C_1$-trees correspond to its edges and $C_2$-trees to its faces, see Figure \ref{conf-canon-form}, where we omit $C_3$-trees. 

\begin{figure}[ht]
$$
\raisebox{-20 pt}{\begin{overpic}[width=130pt]{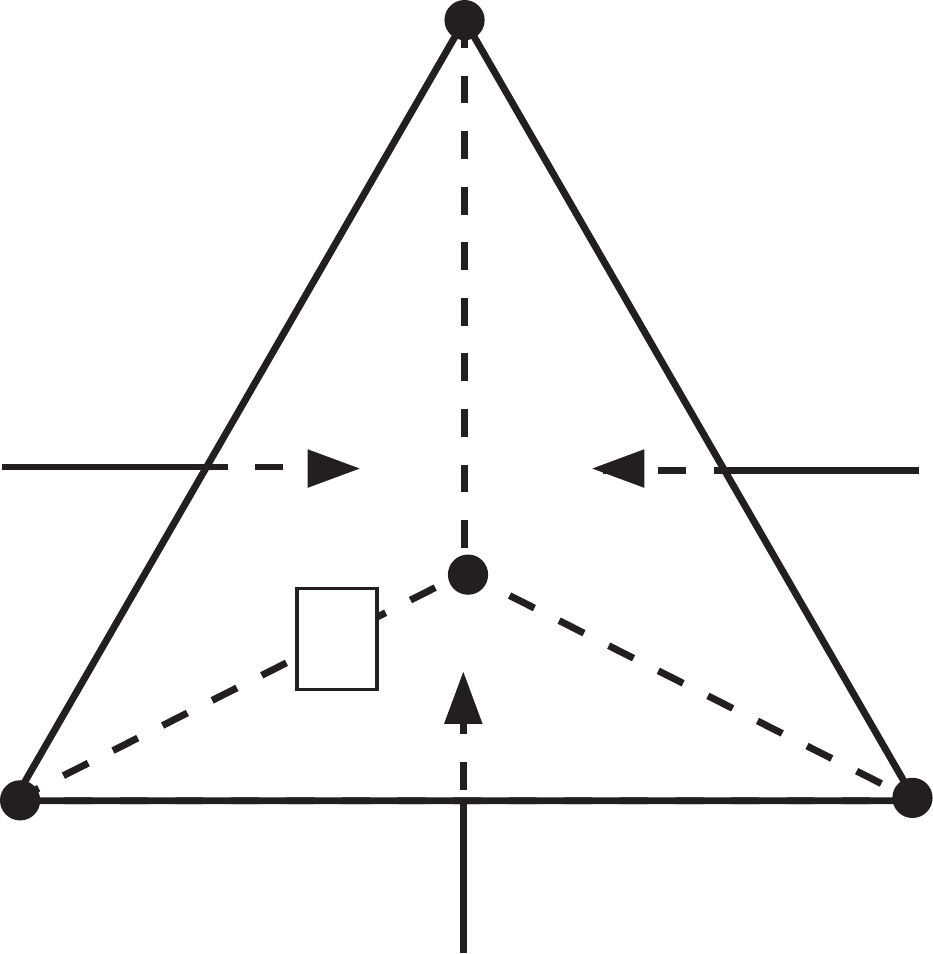}
\put(-6,11){\footnotesize \textbf{$1$}}
\put(131,11){\footnotesize \textbf{$2$}}
\put(70,53){\footnotesize \textbf{$3$}}
\put(69,128){\footnotesize \textbf{$4$}}
\put(94,40){\footnotesize$c_1$}
\put(26,40){\footnotesize$c_2$}
\put(53,14){\footnotesize$c_3$}
\put(21,75){\footnotesize$c_4$}
\put(100,75){\footnotesize$c_5$}
\put(53,83){\footnotesize$c_6$}
\put(129,65){\footnotesize$f_1$}
\put(-10,65){\footnotesize$f_2$}
\put(43,40){\footnotesize$f_3$}
\put(61,-10){\footnotesize$f_4$}
\end{overpic}}$$
\caption{The configuration of claspers of the standard form.} \label{conf-canon-form}
\end{figure}
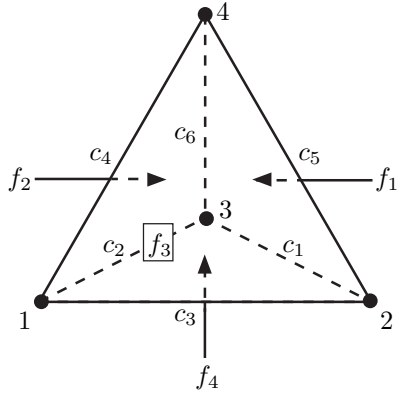

\begin{remark}\label{3-comp-case}
For 2- and 3-component links, we can define standard forms as in Figure \ref{2-3-comp} by forgetting the third and fourth components or the fourth component respectively. 
Note that these forms were shown in \cite{TY}. 
For a 2-component link, the standard form is determined uniquely by the number $c_3$ which is the linking number multiplied by $-1$. Remember that the link-homotopy classes of 2-component links are classified by the linking number. For a 3-component link, we can also determine a unique standard form up to link-homotopy, see Remark \ref{3-comp-psi} below. 
\end{remark}

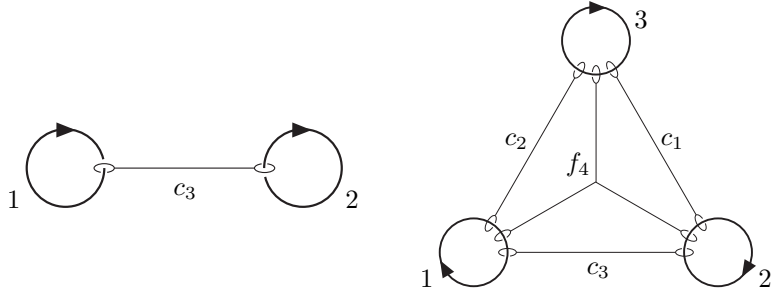
\begin{figure}[ht]
\raisebox{30 pt}{\begin{overpic}[width=120pt]{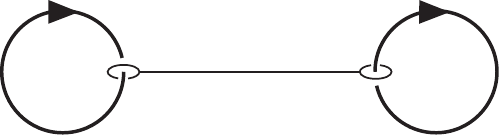}
\put(-7,0){\footnotesize \textbf{$1$}}
\put(121,0){\footnotesize \textbf{$2$}}
\put(56,5){\footnotesize$c_3$}
\end{overpic}}
\hspace{1.0cm}
\raisebox{0 pt}{\begin{overpic}[width=120pt]{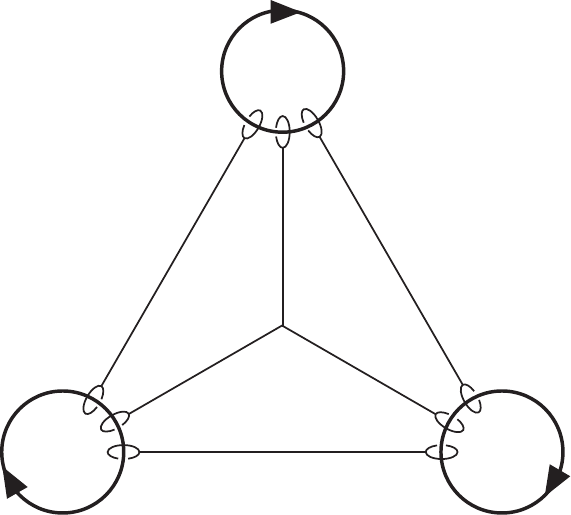}
\put(-7,0){\footnotesize \textbf{$1$}}
\put(121,0){\footnotesize \textbf{$2$}}
\put(74,98){\footnotesize \textbf{$3$}}
\put(84,54){\footnotesize$c_1$}
\put(25,54){\footnotesize$c_2$}
\put(56,5){\footnotesize$c_3$}
\put(48,44){\footnotesize$f_4$}
\end{overpic}}
\caption{The standard forms for 2- and 3-component links.} \label{2-3-comp}
\end{figure}

\section{Proof of Theorem \ref{mainthm}} \label{main}
\par
In this section, we define moves for the standard forms which correspond to the relations $\psi_{ij}$ in Theorem \ref{mainthm} and prove the theorem. We also call the moves $\psi_{ij}$.

\begin{definition} \label{psi-move}
We define the move $\psi_{ij}$ for the standard form of 4-component links, see Figure \ref{psi(ij)-move}, where we present the configuration as a tetrahedron and omit $C_2$- and $C_3$-trees for simplicity. Let $D_i$ be the disc which spans the trivial circle of the $i$-th component. The move $\psi_{ij}$ pushes an arc (marked by $\star$) of the trivial circle of the $j$-th component along the parallel $C_{1}$-trees which connect the $i$-th and $j$-th components, slides it over $D_i$, backs it under $D_i$ and finally pulls it back along the previous $C_{1}$-trees. When we pull back the arc, it goes across the $C_1$- and $C_2$-trees, the move is marked by $(\ast)$, and, from Lemma \ref{edge-comp-change} (see also Example \ref{example-edge-comp-change}), new $C_2$- and $C_3$-trees occur respectively. 
\end{definition}

\begin{figure}[ht]
\raisebox{-44 pt}{\begin{overpic}[width=110pt]{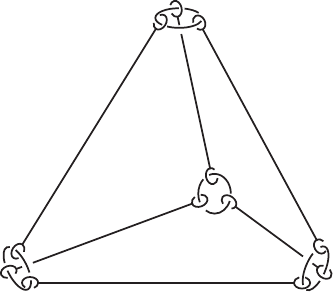}
\put(44,92){$i$}
\put(-3,-9){$j$}
\put(57,34){$k$}
\put(112,10){$l$}
\put(-5,14){\footnotesize $\star$}
\end{overpic}}
\,\,$\longrightarrow$\,\,
\raisebox{-44 pt}{\begin{overpic}[width=110pt]{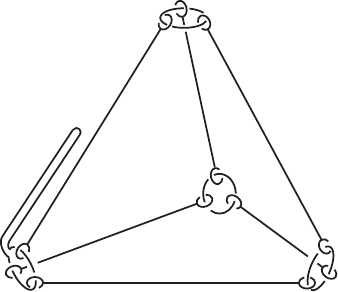}
\end{overpic}}
\,\,$\longrightarrow$\,\,
\raisebox{-44 pt}{\begin{overpic}[width=110pt]{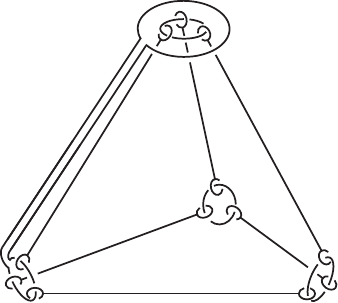}
\end{overpic}}\\

\,\,$\longrightarrow$\,\,
\raisebox{-44 pt}{\begin{overpic}[width=110pt]{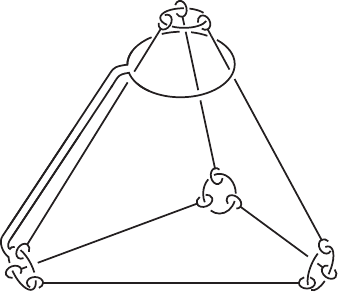}
\end{overpic}}
\,\,$\overset{(\ast)}{\longrightarrow}$\,\,
\raisebox{-44 pt}{\begin{overpic}[width=110pt]{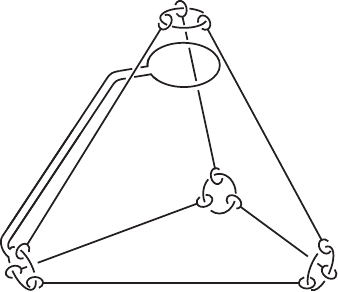}
\end{overpic}}
\,\,$\longrightarrow$\,\,
\raisebox{-44 pt}{\begin{overpic}[width=110pt]{clasper-change13.pdf}
\end{overpic}}
\caption{The move $\psi_{ij}$.} \label{psi(ij)-move} 
\end{figure}

\begin{remark}
The moves similar to $\psi_{ij}$ are shown in \cite[Figure 6]{L} and in \cite[Figure 21]{TY}. 
\end{remark}

The details of the changes of the numbers of claspers under $\psi_{ij}$ are depicted in Figure \ref{psi(ij)-move2}, where they are shown as diagrams, thick lines are components of the trivial link, thin lines are claspers and we depict only claspers which relate to this move. The integers $a$, $b$ and $c$ are the numbers of parallel claspers. Here the second move $(\ast\ast)$ holds up to link-homotopy and from Lemma \ref{edge-comp-change} (see also Example \ref{example-edge-comp-change}) the fourth move $(\ast\ast\ast)$ makes new $C_2$- and $C_3$-trees which are presented by the dotted lines with the numbers of parallel claspers. 

\begin{figure}[ht]
\raisebox{-37 pt}{\begin{overpic}[width=100pt]{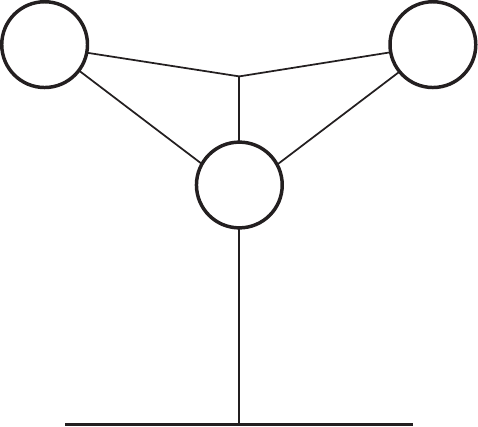}
\put(36,35){$i$}
\put(18,6){$j$}
\put(-9,70){$k$}
\put(102,70){$l$}
\put(26,56){\footnotesize$a$}
\put(72,56){\footnotesize$b$}
\put(48,76){\footnotesize$c$}
\end{overpic}}
\,\,$\longrightarrow$\,\,
\raisebox{-37 pt}{\begin{overpic}[width=100pt]{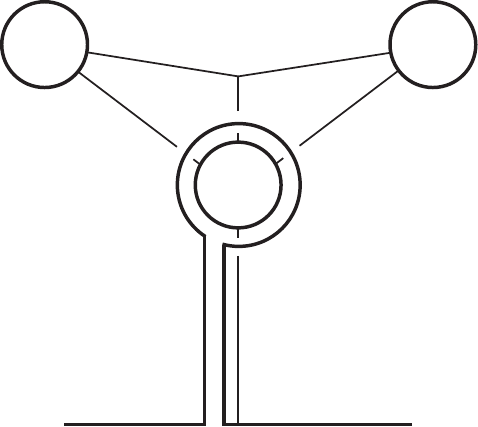}
\end{overpic}}
\,\,$\overset{(\ast\ast)}{\longrightarrow}$\,\,
\raisebox{-37 pt}{\begin{overpic}[width=100pt]{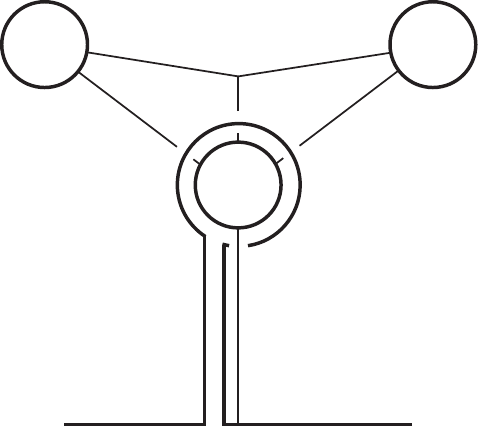}
\end{overpic}}\\

\,\,$\longrightarrow$\,\,
\raisebox{-37 pt}{\begin{overpic}[width=100pt]{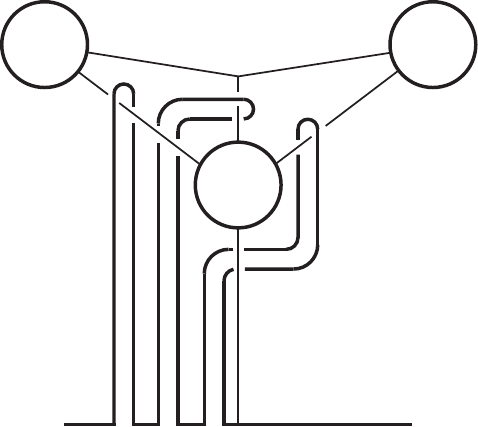}
\end{overpic}}
\,\,$\overset{(\ast\ast\ast)}{\longrightarrow}$\,\,
\raisebox{-37 pt}{\begin{overpic}[width=100pt]{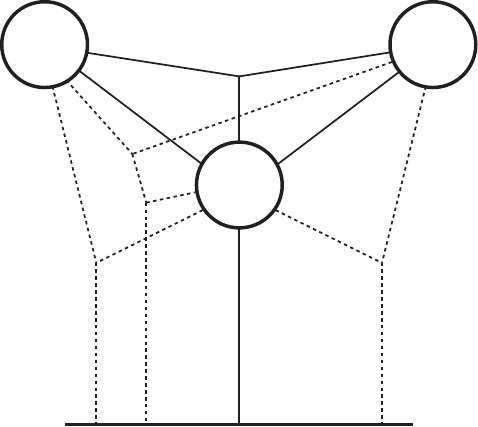}
\put(40,32){$i$}
\put(9,6){$j$}
\put(-9,70){$k$}
\put(102,70){$l$}
\put(4,32){\footnotesize$+a$}
\put(83,32){\footnotesize$-b$}
\put(28,51){\scriptsize$+c$}
\end{overpic}}
\caption{The change of the numbers of claspers by $\psi_{ij}$.} \label{psi(ij)-move2}
\end{figure}
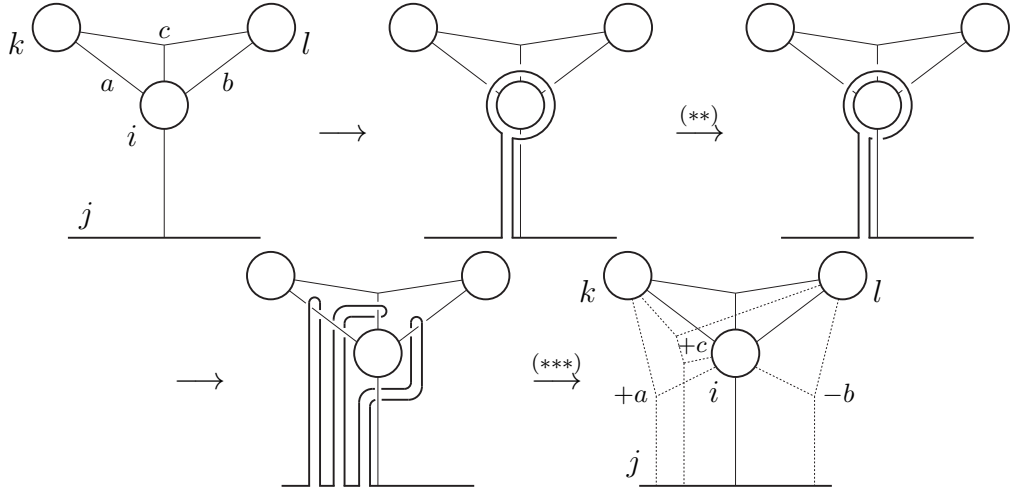

The changes of the numbers of claspers under a part of $\psi_{ij}$ moves are listed in Table \ref{originaltable}. The other moves of $\psi_{ij}$ are obtained from compositions of them. Let $i, j, k, l$ are distinct numbers in $\{1, 2, 3, 4\}$, then $\psi^{-1}_{il}=\psi_{jl}\psi_{kl}$. 

\begin{remark}\label{3-comp-psi}
We can also define moves $\psi'_{ij}$, where $i\neq j$ and $1\leq i, j \leq 3$, for the standard form of a 3-component link (see Figure \ref{2-3-comp}) by forgetting the fourth component from $\psi_{ij}$. 
Let $\mathcal{L}_3$ be the set of the link-homotopy classes of 3-component links and $N'$ the set of 4-tuples of integers $e_1$, $e_2$, $e_3$ and $f_4$ modulo $\psi'_{ij}$. We can check that there is a bijection between $\mathcal{L}_3$ and $N'$, see the proof of Theorem 1.7 in \cite{TY}. Note that, in \cite{TY}, this fact was discussed up to clasp-pass moves (i.e. $C_3$-moves) then that induces the classification up to link-homotopy since a clasp-pass move is obtained by link-homotopy for 3-component links. 
\par 
Moreover, we can put $f_4$ in the range $0 \leq f_4 < \gcd^{\ast} (e_1, e_2, e_3)$ by using $\psi'_{ij}$. Then the number $f_4$ is equal to the Milnor homotopy invariant $\overline{\mu}_L(123)$ and this standard form is uniquely determined for a 3-component link. Remember that the link-homotopy classes of 3-component links are classified by the Milnor homotopy invariants $\overline{\mu}_L(12)$, $\overline{\mu}_L(23)$, $\overline{\mu}_L(13)$ and $\overline{\mu}_L(123)$.
\end{remark}

\par
We recall the argument in \cite{L} (see also \cite{Mil01}). Let $L$ be an $n$-component link and $m_i$ a meridian of the $i$-th component of $L$ in the fundamental group $\pi_1(S^3\setminus L)$ of the complement of $L$. Here $m_i$ is a loop which starts the basepoint to a point near the $i$-th component, goes around the $i$-th component once with linking number 1 and goes back the first path. The \textit{reduced group} $\mathcal{G}(L)$ of $L$ is $\pi_1(S^3\setminus L)/\bigl<\!\bigl<[m_i^g,m_i^{g'}]\bigr>\!\bigr> $, where $[\cdot,\cdot]$ is a commutator: $[a,b]=aba^{-1}b^{-1}$, $\bigl<\!\bigl<[m_i^g,m_i^{g'}]\bigr>\!\bigr>$ is the normal closure of the set $\{[m_i^g, m_i^{g'}]| g, g'\in \pi_1(S^3\setminus L), 1\leq i\leq n\}$, and $m_i^g$ is the conjugate of $m_i$: $m_i^g=gm_ig^{-1}$.
It was shown that $\mathcal{G}(L)$ is generated by meridians $m_i$.
It was also shown that if two links $L$ and $L'$ are link-homotopic then $\mathcal{G}(L)$ and $\mathcal{G}(L')$ are isomorphic. 
\par
Suppose that two ($n+1$)-component links $L$ and $L'$ have the equivalent first $n$-component sublink $L_n$ in common; $L = L_n\cup K$ and $L'=L_n\cup K'$. Consider the reduced group $\mathcal{G}(L_n)$. An automorphism of $\mathcal{G}(L_n)$ which is induced by link-homotopies of $L_n$ is called \textit{a geometric automorphism}. Let $\alpha$ and $\alpha'$ in $\mathcal{G}(L_n)$ represent $K$ and $K'$ respectively. Then it was proved that the two ($n+1$)-component links $L$ and $L'$ are link-homotopic if and only if there is a geometric automorphism $\Phi$ of $\mathcal{G}(L_n)$ satisfying $\Phi(\alpha)=\alpha'$. 

\par
Consider the 4-component case. Let $L$ be a 4-component link and $K_i$ $(1\leq i \leq 4)$ its $i$-th component. The first 3-component sublink $L_3=K_1\cup K_2 \cup K_3$ is classified up to link-homotopy by 4 integers $k$, $l$, $r$ and $d$ $(0\leq d < \gcd^{\ast} (k,l,r) )$, i.e. the Milnor homotopy invariants of length up to three. Let $x$, $y$ and $z$ be meridians of $K_1$, $K_2$ and $K_3$ respectively and $\alpha$ a presentation of $K_4$ in the reduced group $\mathcal{G}(L_3)$ of the sublink $L_3$. 
From algebraic arguments, (since $\mathcal{G}(L_3)$ is a nilpotent group of class $4$ and commutators with repeats vanish from the relations of $\mathcal{G}(L_3)$,) $\alpha$ is presented as follows.
$$\alpha=x^{e_1}y^{e_2}z^{e_3}[x,y]^{e_4}[x,z]^{e_5}[y,z]^{e_6}[y,[x,z]]^{e_7}[z,[x,y]]^{e_8},$$
where $e_i$ is an integer which is called a \textit{commutator number}. This presentation has ambiguities induced from relations of $\mathcal{G}(L_3)$ and geometric automorphisms of $\mathcal{G}(L_3)$.
These ambiguities are summarized in Table \ref{L-table}; $\Phi_1$ and $\Phi_2$ are from the relations of $\mathcal{G}(L_3)$, $\Phi_3$, $\Phi_4$ and $\Phi_5$ are from the inner (geometric) automorphisms of $\mathcal{G}(L_3)$, i.e. mapping meridians to their conjugates, 
and $\Phi_6$ is from the outer geometric automorphisms of $\mathcal{G}(L_3)$.

\begin{remark} \label{red-rel}
We note relations in $\mathcal{G}(L_3)$ with fixed meridians $x$, $y$ and $z$. Let $g$ and $h$ be words in $\{x,y,z\}$ and $g_{\check{a}}$ the word which is obtained by removing $a$ from $g$ for $a \in \{x,y,z\}$.  Then the following relations hold. 
\begin{itemize}
\item $a^{g}a^{h}=a^{h}a^{g}$ for $a\in \{x,y,z\}$. 
\item $a^g=a^{g_{\check{a}}}$ for $a\in \{x,y,z\}$.
\item $[a,b]^{-1}=[a^{-1},b]=[a,b^{-1}]$ and $[a,[b,c]]^{-1}=[a^{-1},[b,c]]=[a,[b,c]^{-1}]$ for $a$, $b$ and $c\in \{x,y,z\}$.  
\end{itemize}
\end{remark}

\begin{proof}[Proof of Theorem \ref{mainthm}] 
Let $L_a=L_a(k,l,r,d,e_1,\dots, e_8)$ be a 4-component link which consists of $L_3$ determined by $(k, l, r, d)$ and the fourth component determined by $\alpha$. A shape of $L_a$ is described by using claspers since the commutators in $\alpha$ correspond to simple tree claspers through surgeries, see Figure \ref{surg-rep}, where $a$, $b$ and $c$ are meridians and the dotted arcs belong to the fourth component in $\mathcal{G}(L_3)$. We can check that by straightforward computations using the relations in Remark \ref{red-rel}. The shape of $L_a$ is as in Figure \ref{Levine-shape}, where we depict $C_3$-trees separately. In the figures, the basepoint $p$ of $\mathcal{G}(L_3)$ is taken above the paper. Let $x$, $y$ and $z$ be meridians of the first, second and third components respectively. The meridians are taken as the loops each of which starts $p$, goes along a straight line connecting $p$ and the segment marked by the arrow to a point near the component, goes around the component with the linking number $+1$ and goes back to $p$ along the straight line. The fourth component $K_4$ is a loop which starts at $p$. However, for simplicity, $K_4$ is moved in the paper by an ambient isotopy. The starting point of $K_4$ is at the arrow mark. 

\begin{figure}[ht]
$$
\raisebox{0 pt}{\begin{overpic}[height=100pt]{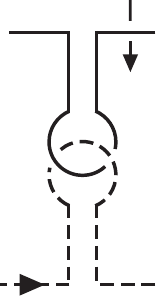}
\put(42,7){\footnotesize$a^{-1}$}\put(36,94){\footnotesize$a$}
\end{overpic}}
\hspace{1.0cm}
\raisebox{0 pt}{\begin{overpic}[height=100pt]{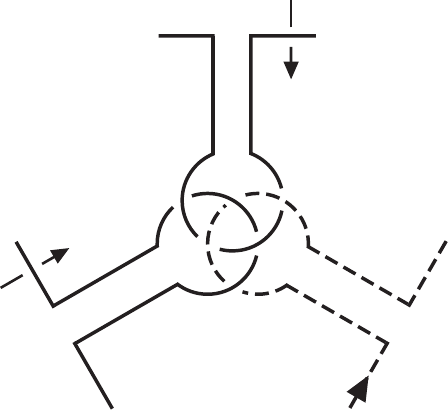}
\put(95,48){\footnotesize$[a,b]$}\put(63,95){\footnotesize$a$}\put(2,22){\footnotesize$b$}
\end{overpic}}
\hspace{1.0cm}
\raisebox{0 pt}{\begin{overpic}[height=100pt]{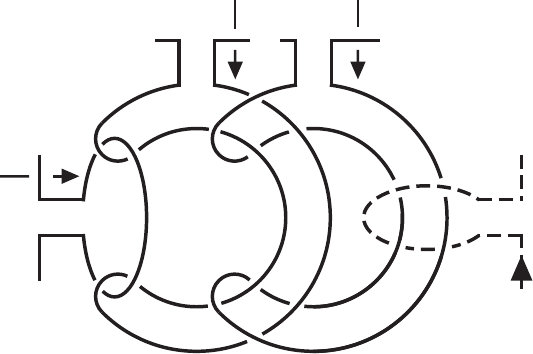}
\put(128,62){\footnotesize$[a,[b,c]]^{-1}$}\put(93,95){\footnotesize$a$}
\put(59,95){\footnotesize$b$}\put(0,42){\footnotesize$c$}
\end{overpic}}
$$\\
\caption{The elements in $\mathcal{G}(L_3)$ corresponding to simple tree claspers.} \label{surg-rep}
\end{figure}

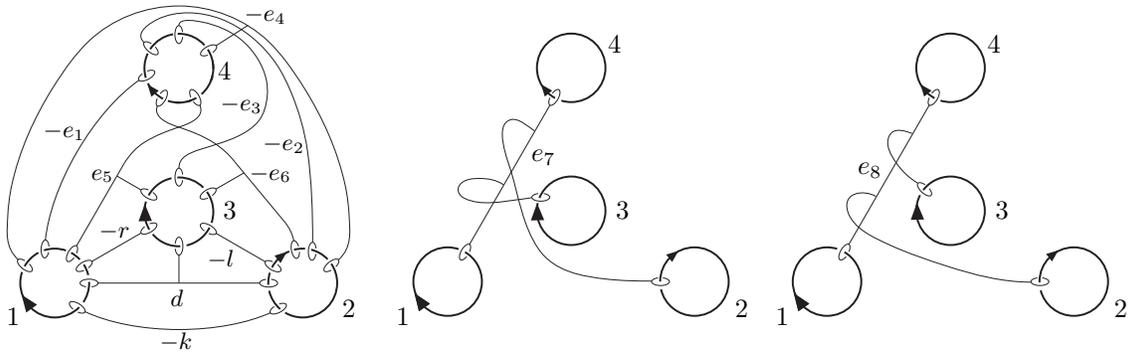
\begin{figure}[ht]
$$
\raisebox{0 pt}{\begin{overpic}[width=130pt]{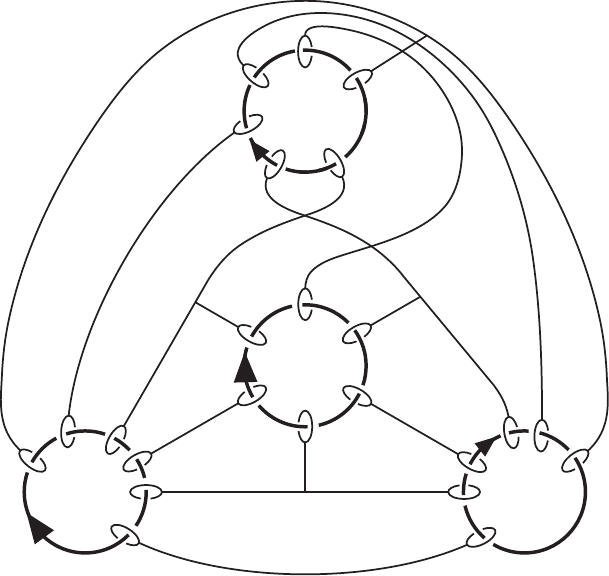}
\put(0,2){\footnotesize \textbf{$1$}}
\put(127,5){\footnotesize \textbf{$2$}}
\put(82,42){\footnotesize \textbf{$3$}}
\put(80,95){\footnotesize \textbf{$4$}}
\put(76,24){\scriptsize$-l$}
\put(35,35){\scriptsize$-r$}
\put(58,-7){\scriptsize$-k$}
\put(14,73){\scriptsize$-e_1$}
\put(97,68){\scriptsize$-e_2$}
\put(81,83){\scriptsize$-e_3$}
\put(92,57){\scriptsize$-e_6$}
\put(32,57){\scriptsize$e_5$}
\put(90,117){\scriptsize$-e_4$}
\put(62,9){\scriptsize$d$}
\end{overpic}}
\qquad
\raisebox{5 pt}{\begin{overpic}[width=120pt]{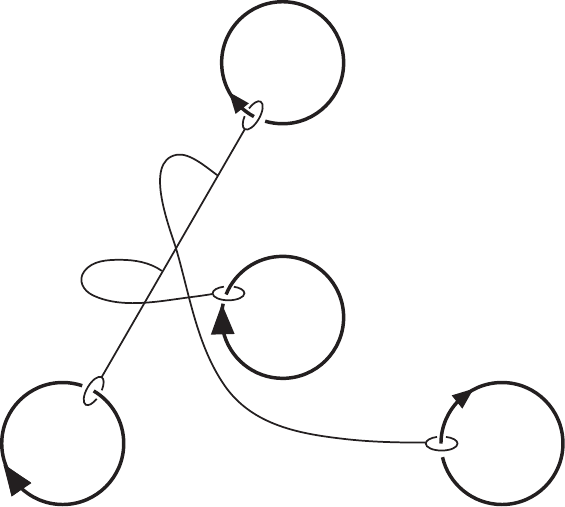}
\put(-6,-3){\footnotesize \textbf{$1$}}
\put(122,0){\footnotesize \textbf{$2$}}
\put(77,37){\footnotesize \textbf{$3$}}
\put(74,100){\footnotesize \textbf{$4$}}
\put(45,59){\footnotesize$e_7$}
\end{overpic}}
\qquad
\raisebox{5 pt}{\begin{overpic}[width=120pt]{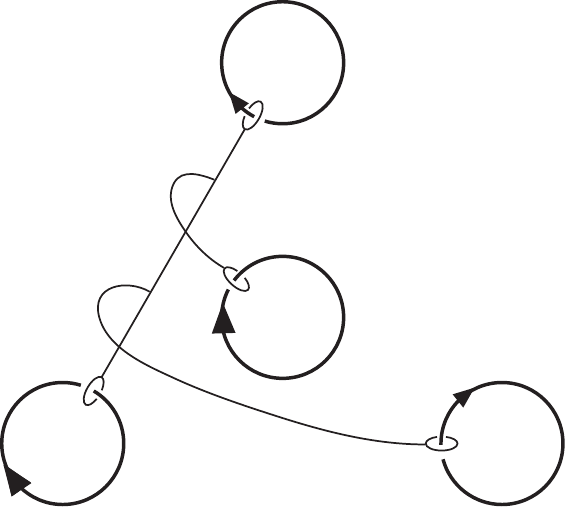}
\put(-6,-3){\footnotesize \textbf{$1$}}
\put(122,0){\footnotesize \textbf{$2$}}
\put(77,37){\footnotesize \textbf{$3$}}
\put(74,100){\footnotesize \textbf{$4$}}
\put(25,54){\footnotesize$e_8$}
\end{overpic}}
$$
\caption{A shape of $L_a$.} \label{Levine-shape}
\end{figure}

Let $L_b=L_b(c_1,\dots,c_6, f_1,\dots,f_4,t_1,t_2)$ be a 4-component link in the standard form in Figure \ref{canon-form}. The difference between Figure \ref{Levine-shape} and Figure \ref{canon-form} is the positions of the leaves of the claspers and the shapes of $C_3$-trees. By using Lemmas \ref{leaf-exchange}, \ref{IHX} and other lemmas, it is shown that
\begin{eqnarray}
& &L_a(k,l,r,d,e_1,\dots, e_8)  \nonumber \\ 
&\overset{\mbox{\scriptsize l.h.}}{\sim}&L_b(-l,-r,-k,-e_1,-e_2,-e_3,-e_6,e_5,-e_4,d,e_7+e_8+e_3e_4, -e_7). \nonumber
\end{eqnarray}
By using $[\psi_{12},\psi_{41}]$ (see Table \ref{tablecase} below), we can simplify the eleventh argument of $L_b$ up to link-homotopy, 
$$L_b(\cdots ,e_7+e_8+e_3e_4, \cdots)\overset{\mbox{\scriptsize l.h.}}{\sim}L_b(\cdots ,e_7+e_8, \cdots).$$
From the observation, we define a map $\mathcal{F}:M\to N$ by 
$$(k,l,r,d,e_1,\dots, e_8)\mapsto (-l,-r,-k,-e_1,-e_2,-e_3,-e_6,e_5,-e_4,d,e_7+e_8, -e_7).$$
Then, $L_a(T)\overset{\mbox{\scriptsize l.h.}}{\sim}L_b(\mathcal{F}(T))$ for any $T\in M$.
We show that $\mathcal{F}$ is well-defined. For $T=(k,l,r,d,e_1,\dots, e_8) \in M$, put $\mathcal{F}(T)=(c_1,\dots,c_6,f_1,\cdots,f_4,t_1,t_2) \in N$. Then the following relations hold.  
\begin{equation*} k= -c_3,\, l= -c_1,\, r= -c_2,\, e_1= -c_4,\, e_2= -c_5,\, e_3= -c_6, \label{eq02}
\end{equation*}
$$d= f_4,\, e_4= -f_3,\, e_5= f_2,\, e_6= -f_1,\, e_7= -t_2\ \mbox{ and } e_8= t_1+t_2.$$
Under these relations, we can check that the relations $\psi_{ij}$ in Table \ref{originaltable} for $L_b$ generate the relations $\Phi_{i}$ in Table \ref{L-table} for $L_a$ as follows.
$$\Phi_1=\psi_{14},\, \Phi_2=(\psi_{14}\psi_{34})^{-1},\, \Phi_3=\psi_{12}\psi_{32},\,\Phi_4=\psi_{41}^{-1},$$
$$\Phi_5=\psi_{43}^{-1}\,\mbox{ and }\Phi_6(a,b,c)=\psi_{21}^c\psi_{12}^b\{(\psi_{23}\psi_{43})^{-1}\}^a.$$
Thus, if $T$, $T' \in M$ are related by $\Phi_i$, then $\mathcal{F}(T)$, $\mathcal{F}(T')\in N$ are related by $\psi_{ij}$.  
This induces well-definedness of $\mathcal{F}$. The map $\mathcal{F}$ is obviously surjective.
\par
We show that $\mathcal{F}$ is injective. 
Assume $\mathcal{F}(T)=\mathcal{F}(T')$ in $N$ for $T$, $T'\in M$, i.e. $\mathcal{F}(T)$ and $\mathcal{F}(T')$ are related by $\psi_{ij}$ in Table \ref{originaltable}, then $L_b(\mathcal{F}(T)) \overset{\mbox{\scriptsize l.h.}}{\sim} L_b(\mathcal{F}(T'))$. Since $L_a(T) \overset{\mbox{\scriptsize l.h.}}{\sim} L_b(\mathcal{F}(T))$ and $L_a(T') \overset{\mbox{\scriptsize l.h.}}{\sim} L_b(\mathcal{F}(T'))$, $L_a(T) \overset{\mbox{\scriptsize l.h.}}{\sim} L_a(T')$ holds. From Theorem \ref{Levine}, $T=T'$ in $M$. Thus $\mathcal{F}$ is injective. 
\par 
Combining with Theorem \ref{Levine}, the bijectivity of $\mathcal{F}$ induces that there is a bijection between $\mathcal{L}_4$ and $N$. 
\end{proof}

\begin{remark}
A map $(c_1,\dots,c_6,f_1,\dots,f_4,t_1,t_2)\mapsto L_b(c_1,\dots,c_6,f_1,\dots,f_4,t_1,t_2)$ gives a bijection from $N$ to $\mathcal{L}_4$.
\end{remark}


In Theorem \ref{Levine}, the integers $k, l, r$ and $d$ $(0\leq d < \gcd^{\ast} (k,l,r))$ are fixed and so is the shape of first 3-component sublink $L_3$ of a corresponding 4-component link $L_a$. Meanwhile, in Theorem \ref{mainthm}, we weaken the condition. Then the shape of $L_3$ can change. That allows us more flexible treatment of the link-homotopy classes $\mathcal{L}_4$ of 4-component links which has higher symmetry with respect to the components, see Table \ref{originaltable}. In Section \ref{classif}, we apply this result to classify subsets of $\mathcal{L}_4$.
\par
We also remark the relations between the numbers $c_i$, $f_i$ and $t_i$ of claspers and the Milnor homotopy invariants. Remark \ref{intro-remark} gives the relation between the 12-tuples $k, l, r, d$ and $e_i$ and the Milnor homotopy invariants $\overline{\mu}_L(I)$ for sequences $I$ of a 4-component link $L$. 
Moreover, as mentioned in \cite{L}, from the relations of the Milnor homotopy invariants \cite{Mil02}, all the other Milnor homotopy invariants for $L$ are calculated from $k, l, r, d$ and $e_i$. 
Then with the relations in Remark \ref{intro-remark2}, we can calculate all the Milnor homotopy invariants through the numbers  $c_i$, $f_i$ and $t_i$ of claspers. Thus the procedure in Section \ref{clasp} gives the schematic way to calculate the Milnor homotopy invariants for 4-component links. 

\section{Classification and algorithm} \label{classif}

\subsection{Invariants} \label{invariants}
\par
In this subsection, we give several new subsets of $\mathcal{L}_4$ which are classified by comparable invariants. Before that, we review Levine's results. 

\begin{theorem}[\cite{L}] \label{Levintheorem} 
The subsets of $\mathcal{L}_4$ satisfying the following conditions are classified by comparable invariants. 
\begin{enumerate}
\item $k=l=r=e_1=e_2=e_3=0$ $($i.e. all linking numbers vanish.$)$ \\
\item $k=l=r=e_1=e_2=0$ and $e_3\neq 0$ $($i.e. all but one of linking numbers vanish.$)$ \\
\item $l=r=e_1=e_2=0$ and $k, e_3 \neq 0$ $($i.e. all but two of linking numbers vanish and the two linking numbers correspond to opposite sides in the tetrahedron.$)$\\
\item $e_1$, $e_2$ and $e_3$ are pairwise relatively prime $($i.e. three linking numbers corresponding to three edges contacting a vertex in the tetrahedron are pairwise relatively prime.$)$ \\
\item $k$, $l$ and $r$ are pairwise relatively prime $($i.e. three linking numbers corresponding to three edges bounding a face in the tetrahedron are pairwise relatively prime.$)$
\end{enumerate}
\end{theorem}

\begin{remark} \label{classrem}
In Theorem \ref{Levintheorem}, we can represent a complete set of invariants by using our notation for each subset.
\par
In case (1), the complete set of invariants is $f_i$ $(1 \leq i \leq 4)$ and $t_j\pmod{ \gcd_{\ast} (f_1, f_2, f_3, f_4)}$ $(j=1, 2)$. 
\par
In case (2), if $c_6\neq 0$, the complete set of invariants is $c_6$, $f_1\pmod{ c_6}$, $f_2\pmod{ c_6}$, $f_3$, $f_4$, $t_1\pmod{ \gcd (c_6, f_1, f_2, f_3, f_4)}$ and $\Delta= f_1f_2+c_6t_2\pmod{ c_6 \gcd_{\ast} (f_3, f_4)}$.
\par
In case (3), if $c_3, c_6 \neq 0$, the complete set of invariants is $c_3$, $c_6$, $f_1\pmod{ c_6}$, $f_2\pmod{ c_6}$, $f_3\pmod{ c_3}$, $f_4\pmod{ c_3}$, $t_1\pmod{ \gcd (c_3, c_6, f_1, f_2, f_3, f_4)}$ and $\Delta'= c_3c_6t_2+c_3f_1f_2+c_6f_3f_4$. 
\par
In case (4), if $c_4$, $c_5$ and $c_6$ are pairwise relatively prime, the complete set of invariants is $c_1, \dots, c_6$, $f_4 \pmod{ \gcd_{\ast}(c_1, c_2, c_3)}$ and $\theta = c_4f_1+c_5f_2+c_6f_3+\alpha c_5c_6+\beta c_4c_6+\gamma c_4c_5 \pmod{\gcd_{\ast} (c_1c_4-c_2c_5, c_1c_4-c_3c_6, ac_5c_6+bc_4c_6+cc_4c_5)}$, where $(\alpha, \beta, \gamma)$ is a tuple of integers satisfying $0\leq f_4-(\alpha c_1+\beta c_2 +\gamma c_3) < \gcd^{\ast}(c_1, c_2, c_3)$ and $a$, $b$ and $c$ run over all integers satisfying $a c_1+bc_2+c c_3=0$. 
\par
In case (5), if $c_1$, $c_2$ and $c_3$ are pairwise relatively prime, the complete set of invariants is $c_1, \dots, c_6$ and $\theta' = c_1c_2f_3+c_1c_3f_2+c_2c_3f_1+c_2c_5f_4 \pmod{\gcd_{\ast} (c_1c_4-c_2c_5, c_1c_4-c_3c_6)}$.
\end{remark}

We give new classifications for some subsets of $\mathcal{L}_4$ by comparable invariants as a corollary of Theorem \ref{mainthm}. We list the changes of the numbers of the claspers by commutators of $\psi_{ij}$ in Table \ref{tablecase}. The other commutators do not change the numbers of claspers or do the same or inverse change as the one of the commutators in the table.
\begin{table}[htb] 
\begin{center}
   \caption{Changes of the numbers of the claspers by commutators.
   }\label{tablecase}
  \begin{tabular}{|c|c|c|c|c|c|c|} \hline
      & $f_1$ & $f_2$ & $f_3$ & $f_4$ & $t_1$ & $t_2$ \\ \hline 
    $[\psi_{14}, \psi_{21}]$ & 0 & $0$ & $0$ & $0$ & $0$ & $c_1$ \\ \hline 
    $[\psi_{14}, \psi_{12}]$ & 0 & $0$ & $0$ & $0$ & $c_2$ & $-c_2$ \\ \hline 
    $[\psi_{43}, \psi_{14}]$ & 0 & $0$ & $0$ & $0$ & $c_3$ & $0$ \\ \hline 
    $[\psi_{32}, \psi_{43}]$ & 0 & $0$ & $0$ & $0$ & $0$ & $c_4$ \\ \hline     
    $[\psi_{21}, \psi_{23}]$ & 0 & $0$ & $0$ & $0$ & $-c_5$ & $c_5$ \\ \hline 
    $[\psi_{12}, \psi_{41}]$ & 0 & $0$ & $0$ & $0$ & $c_6$ & $0$ \\ \hline 
  \end{tabular} \\
\end{center}
\end{table}

\begin{proposition} \label{case1}  
The subsets of $\mathcal{L}_4$ satisfying the following conditions are classified by the numbers of claspers. 
\begin{enumerate}
\item Let $c_2=c_4=c_5=c_6=0$, $c_1, c_3 \neq 0$, $f_1 \equiv 0\pmod{c_1}$ and $f_3 \equiv 0\pmod{ c_3}$.
Then a complete set of invariants is
$$c_1,\,c_3,\, f_2,\, f_4\,\, (\bmod\, \mbox{$\gcd$} (c_1, c_3)),$$ 
$$\Delta_1=c_1t_1+f_1f_4\,\, (\bmod\, c_1\mbox{$\gcd$} (c_3, f_2))$$
and 
$$\Delta_2=c_3t_2+f_3f_4\,\, (\bmod\, c_3\mbox{$\gcd$} (c_1, f_2)).$$ 
\item Let $c_2=c_4=c_5=c_6=0$, $c_1=c_3\neq 0$, $f_1\equiv 0\pmod{\gcd (c_1, f_2, f_3)}$ and $f_3\equiv 0\pmod{\gcd (c_1, f_1, f_2)}$.  
Then a complete set of invariants is
$$c_1,\, f_1\, (\bmod\, c_1),\, f_3\, (\bmod\, c_1),\, f_2,\, f_4\, (\bmod\, c_1),$$ 
$$\Delta_1=c_1t_1+f_1f_4\,\, (\bmod\, c_1 \mbox{$\gcd$} (c_1, f_1, f_2, f_3))$$
and 
$$\Delta_2=c_1t_2+f_3f_4\,\, (\bmod\, c_1 \mbox{$\gcd$} (c_1, f_1, f_2, f_3)).$$
\item Let $c_2=c_5=c_6=0$, $c_1$, $c_3$, $c_4\neq 0$ and $\gcd(c_1,c_3)=\gcd(c_3,c_4)=1$.  Then a complete set of invariants is
$$c_1,\, c_3,\, c_4,\, f_1\, (\bmod\, c_1),\, f_2\, (\bmod\, c_4),$$ 
$$\Delta_2=c_3t_2+f_3f_4\,\, (\bmod\, \mbox{$\gcd$} (c_1, c_4, c_3f_1, c_3f_2))$$
and
$$\Delta_3 = c_1c_4t_1+c_4f_1f_4+c_1f_2f_3\pmod{ c_3}.$$
\item Let $c_3=c_6=0$, $c_1, c_2, c_4, c_5 \neq 0$ and $\gcd (c_1, c_4) = \gcd (c_2, c_5) = 1$. Then a complete set of invariants is
$$c_1,\, c_2,\, c_4,\, c_5,\, f_1\,\, (\bmod\, \gcd(c_1, c_5)),\, f_2\,\, (\bmod\, \gcd(c_2, c_4)),\, f_3\,\, (\bmod\, \gcd(c_4, c_5))$$
and 
$$f_4\,\, (\bmod\, \gcd(c_1, c_2))\quad\mbox{(i.e. the Milnor homotopy invariants.)}$$ 
\item Let $c_5=c_6=0$, $c_1, c_2, c_3, c_4 \neq 0$ and $\gcd (c_1, c_4) =\gcd (c_2, c_3) =\gcd(c_2, c_4)=1$ (or $\gcd (c_1, c_3) =\gcd (c_2, c_3) =\gcd(c_2, c_4)=1$). Then a complete set of invariants is 
$$c_1,\, c_2,\, c_3,\, c_4,\, f_1\,\, (\bmod\, c_1)
\mbox{ and }
\Delta_4 = c_1c_2f_3+c_1c_3f_2+c_2c_3f_1\,\, (\bmod\, c_1c_4).$$ 
\end{enumerate}
\end{proposition}

\begin{remark} 
We also give complete sets of invariants for other cases obtained by using the symmetry of the tetrahedron and Lemma \ref{IHX} if necessary.
\end{remark}

\begin{proof}
For each case, let $L$ and $L'$ be 4-component links with the assumptions and the same values of these invariants. We then give a geometric transformation between $L$ and $L'$ by using the moves $\psi_{ij}$. 
\par
(1) By $\psi_{34}$ and $\psi_{14}$, we transform $f_1$ and $f_3$ of both $L$ and $L'$ into 0, and by $\psi_{21}$ and $\psi_{23}$, we transform $f_4$ of both into an integer $\bar{f}_4$ ($0 \leq \bar{f}_4 < \gcd (c_1, c_3)$). 
Then $\Delta_1=c_1t_1 \pmod{c_1 \gcd (c_3, f_2)}$ and $\Delta_2 =c_3t_2 \pmod{c_3 \gcd (c_1, f_2)}$ of both are the same, because $\Delta_1$ and $\Delta_2$ are invariants under $\psi_{ij}$. 
Therefore $t_1\pmod{ \gcd(c_3, f_2)}$ and $t_2\pmod{ \gcd (c_1, f_2)}$ of both are the same.
By the 3rd line of Table \ref{tablecase} and $\psi_{12}$, we can transform $t_1$ of $L$ into that of $L'$ without affecting the other elements.
Similarly, by the 1st line of Table \ref{tablecase} and $\psi_{32}$, we can transform $t_2$ of $L$ into that of $L'$ without affecting the other elements. 
\par
(2) By $\psi_{34}$, $\psi_{14}$ and $\psi_{21}$, we transform $f_1$, $f_3$ and $f_4$ of both into integers $\bar{f}_1$, $\bar{f}_3$ and $\bar{f}_4$ ($0 \leq \bar{f}_1, \bar{f}_3, \bar{f}_4 < c_1$) respectively.
Then $\Delta_1 =c_1t_1+\bar{f}_1\bar{f}_4$ and $\Delta_2=c_1t_2+\bar{f}_3\bar{f}_4$ of both are the same modulo ${c_1 \gcd (c_1, \bar{f}_1, f_2, \bar{f}_3)}$ respectively. 
Therefore  $t_1$ and $t_2$ of both are the same modulo ${\gcd (c_1, \bar{f}_1, f_2, \bar{f}_3)}={\gcd (c_1, f_2, \bar{f}_3)}={\gcd (c_1, \bar{f}_1, f_2)}$, respectively. 
By the 3rd line of Table \ref{tablecase}, $\psi_{12}$ and $\psi_{43}$, we can transform $t_1$ of $L$ into that of $L'$ without affecting the other elements.
Similarly, by the 1st line of Table \ref{tablecase}, $\psi_{41}$ and $\psi_{32}$, we can transform $t_2$ of $L$ into that of $L'$ without affecting the other elements. 
\par
(3) By $\psi_{34}$, we transform $f_1$ of both $L$ and $L'$ into an integer $\bar{f}_1$ ($0 \leq \bar{f}_1 < c_1$), by $\psi_{43}$, $f_2$ into $\bar{f}_2$ ($0 \leq \bar{f}_2 < c_4$) and by $\psi_{12}$, $\psi_{14}$, $\psi_{21}$ and $\psi_{23}$, $f_3$ and $f_4$ into 0. 
Then, $\Delta_2= c_3t_2 \pmod{\gcd (c_1,c_4,c_3\bar{f}_1, c_3\bar{f}_2)}$ and $\Delta_3=c_1c_4t_1\pmod {c_3}$ of both are the same respectively.
By the assumption, two $t_2$ are the same modulo $\gcd (c_1,c_4,c_3\bar{f}_1, c_3\bar{f}_2)$ and two $t_1$ are the same modulo $c_3$. 
Therefore, by using the 3rd line of Table \ref{tablecase} and $\psi_{41}$, $\psi_{32}$ and the 1st and 4th lines of Table \ref{tablecase}, we can transform $t_1$ and $t_2$ of $L$ into those of $L'$ without affecting the other elements respectively. 
\par
(4) By $\psi_{23}$ and $\psi_{34}\psi_{14}$, we transform $f_1$ of both $L$ and $L'$ into an integer $\bar{f}_1$ ($0 \leq \bar{f}_1 < \gcd (c_1,c_5)$) without affecting the other elements except for $t_1$ and $t_2$.
Similarly, we transform $f_2$, $f_3$ and $f_4$ of both into integers $\bar{f}_2$ ($0 \leq \bar{f}_2 < \gcd (c_2,c_4)$), $\bar{f}_3$ ($0 \leq \bar{f}_3 < \gcd (c_4,c_5)$) and $\bar{f}_4$ ($0 \leq \bar{f}_4 < \gcd (c_1,c_2$)) respectively, without affecting the other elements except for $t_1$ and $t_2$.
Finally, by the 1st, 2nd, 4th and 5th lines of Table \ref{tablecase}, we transform $t_1$ and $t_2$ of $L$ into those of $L'$ respectively, without affecting the other elements. 
\par
(5) By $\psi_{34}$, we transform $f_1$ of both $L$ and $L'$ into $\bar{f}_1$ ($0 \leq \bar{f}_1 < c_1$)
and by $\psi_{43}$ and $\psi_{14}$, $f_2$ into $0$ without affecting $\bar{f}_1$. 
Then, $\Delta_4=c_1c_2f_3+c_2c_3\bar{f}_1$ of both are the same modulo $c_1c_4$.
Therefore these two $c_2f_3$  are the same modulo $c_4$.
By the assumption $\gcd(c_2, c_4)=1$, these two $f_3$ are the same modulo $c_4$.
Therefore we can transform $f_3$ of $L$ into that of $L'$ by using $\psi_{12}$ without affecting the other elements except for $f_4$, $t_1$ and $t_2$. 
We then transform $f_4$ of $L$ into that of $L'$ by using $\psi_{32}$ and $\psi_{23}$ without affecting the other elements except for $t_1$ and $t_2$. 
Finally, by the 1st, 2nd, 3rd and 4th lines of Table \ref{tablecase} we transform $t_1$ and $t_2$ of $L$ into those of $L'$ respectively, without affecting the other elements.
\end{proof}

The lemmas in Section \ref{clasp} allow us a schematic calculation of link-homotopy classes. 

\begin{example}
In \cite{Hu}, Hughes showed that there are two 4-component links $H_1$ and $H_2$ which have the isomorphic \textit{pre-peripheral structures} \cite{Hu} or \textit{reduced peripheral systems} \cite{AM} (this induces that they are not distinguished by the Milnor homotopy invariants), but they are not link-homotopic.  We show that $H_1$ and $H_2$ are not link-homotopic through the technique in Section \ref{clasp}.
\par
The shapes of $H_1$ and $H_2$ are the closures of the string links which are depicted in the appendix of \cite{AM}. Through the technique in Section \ref{clasp}, we have clasper presentations of the string links corresponding to $H_1$ and $H_2$ as in Figure \ref{Hughes-example}, where the left-most integers are the orders of the components and the claspers wihtout an integer are single claspers.

\begin{figure}[ht] 
$$H_1: 
\hspace{0.5cm}
\raisebox{-28 pt}{\begin{overpic}[width=120pt]{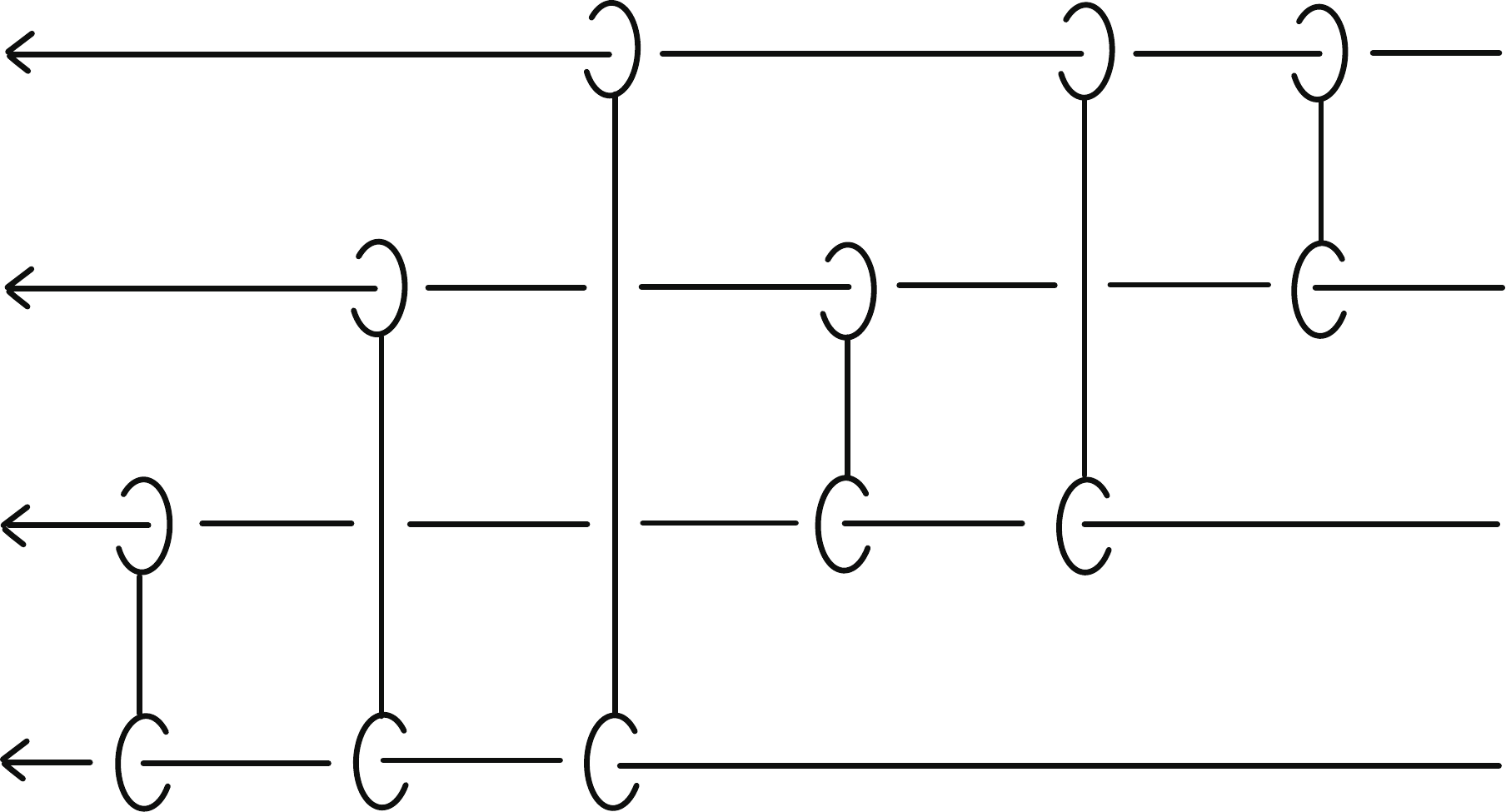}
\put(-8,57){\footnotesize \textbf{$1$}}
\put(-8,38){\footnotesize \textbf{$2$}}
\put(-8,19){\footnotesize \textbf{$3$}}
\put(-8,0){\footnotesize \textbf{$4$}}
\put(90,48){\footnotesize $4$}
\put(109,48){\footnotesize $4$}
\end{overpic}}
\hspace{1.0cm}
H_2: 
\hspace{0.5cm}
\raisebox{-28 pt}{\begin{overpic}[width=140pt]{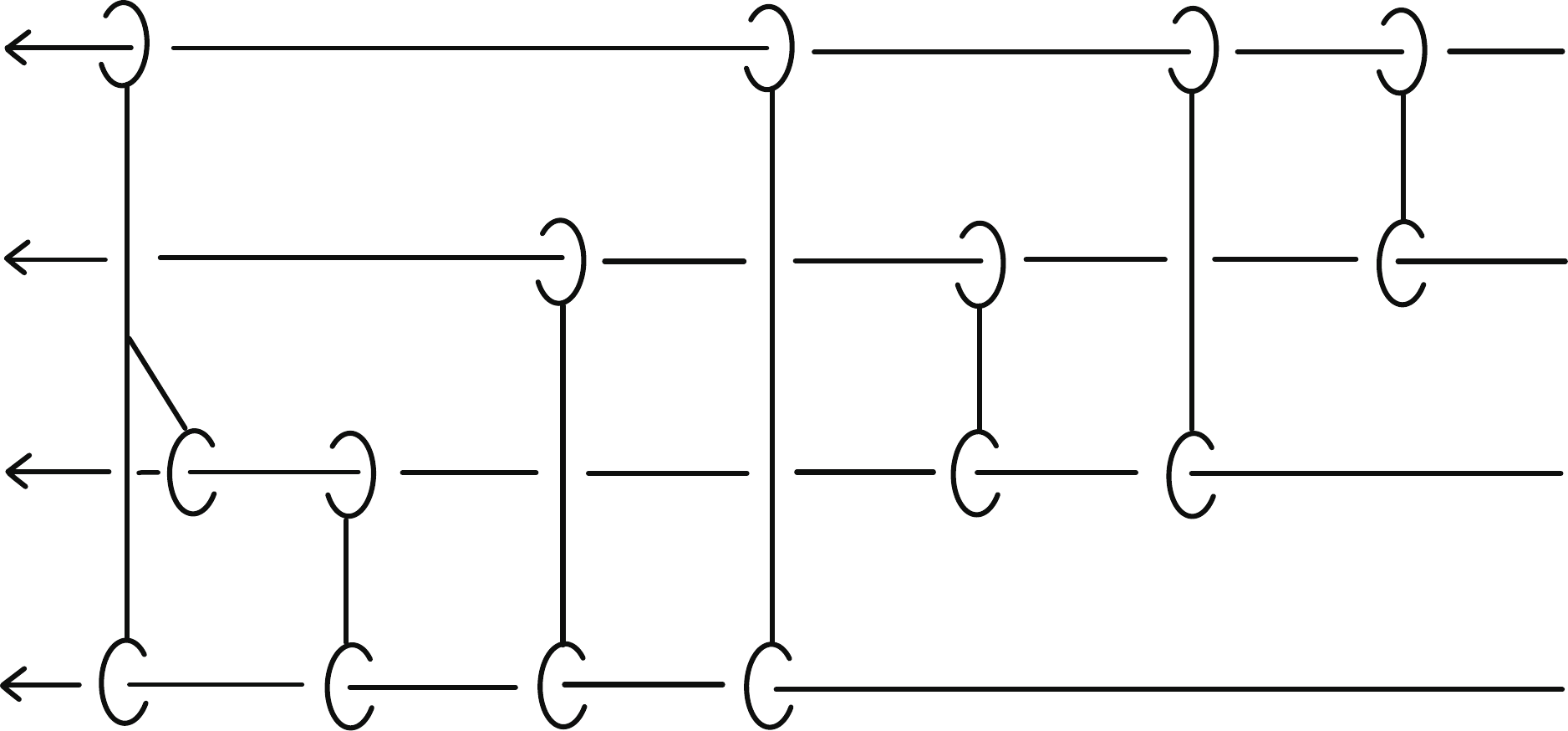}
\put(-8,57){\footnotesize \textbf{$1$}}
\put(-8,38){\footnotesize \textbf{$2$}}
\put(-8,19){\footnotesize \textbf{$3$}}
\put(-8,0){\footnotesize \textbf{$4$}}
\put(15,48){\footnotesize \textbf{$4$}}
\put(110,48){\footnotesize \textbf{$4$}}
\put(129,48){\footnotesize \textbf{$4$}}
\end{overpic}}
$$
\caption{Clasper presentations of the string links corresponding to $H_1$ and $H_2$.} \label{Hughes-example}
\end{figure}
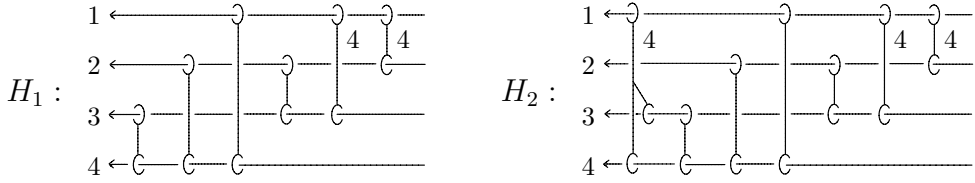
 The both links $H_1$ and $H_2$ have the same linking numbers $-c_1=1$, $-c_2=4$, $-c_3=4$, $-c_4=1$, $-c_5=1$ and $-c_6=1$. From Table \ref{originaltable}, the linking numbers and $\tilde{\theta} = f_1+f_2+f_3+f_4$ ($\bmod$ $3$) gives a complete set of invariants. Here we ignore the $t_1$ and $t_2$ because they vanish from the commutators in Table \ref{tablecase}. (Note that the links $H_1$ and $H_2$ are in the case Theorem \ref{Levintheorem} (4) and from Remark \ref{classrem} (4), we have the invariant $\theta=-\tilde{\theta}$ by $\alpha = -f_4$, $\beta=\gamma=0$.)
\par
From the difference of the numbers of the $C_2$-trees in Figure \ref{Hughes-example}, 
we have $\theta(H_1)-\theta(H_2) \equiv 1\pmod 3$ and $H_1$ and $H_2$ are not link-homotopic. In fact, by using the lemmas in Section \ref{clasp}, we can transform $H_1$ and $H_2$ to the standard form in Figure \ref{canon-form}. One standard form for each of them is $f_1=1$, $f_2=4$, $f_3=4$ and $f_4=16$ for $H_1$ and $f_1=1$, $f_2=8$, $f_3=4$ and $f_4=16$ for $H_2$ respectively. Then $\theta (H_1)=1$ and $\theta (H_2)=2$. 
\end{example}

\subsection{An algorithm} \label{algorithm}
Habegger and Lin \cite{HL} showed that there is an algorithm which determines whether given two links are link-homotopic or not. We translate the algorithm into our notation for the 4-component case. 
Consider the two links $L_1$ and $L_2$ which are represented by the integers as 
$$L_1: c^1_1, c^1_2, c^1_3, c^1_4, c^1_5, c^1_6 \mid f^1_1, f^1_2, f^1_3, f^1_4 \mid t^1_1, t^1_2$$ 
and
$$L_2: c^2_1, c^2_2, c^2_3, c^2_4, c^2_5, c^2_6 \mid f^2_1, f^2_2, f^2_3, f^2_4 \mid t^2_1, t^2_2.$$ 


\begin{enumerate}
\renewcommand{\labelenumi}{\arabic{enumi}.}
\item Check $c^1_i=c^2_i$ ($1\leq i \leq 6$). If they all hold, then go to Step 2, otherwise $L_1$ and $L_2$ are not link-homotopic. 
\item Let $\Psi_1=\psi_{14}^{a_8}\psi_{34}^{a_7}\psi_{23}^{a_6}\psi_{43}^{a_5}\psi_{32}^{a_4}\psi_{12}^{a_3}\psi_{41}^{a_2}
\psi_{21}^{a_1}$ $(a_i\in \mathbb{Z})$ be the composition of $\psi_{ij}$ in Table \ref{originaltable}. Find a tuple $(a_1, \dots, a_8)$ which satisfies $\Psi_1(f^1_i)=f^2_i$ ($1\leq i \leq 4$). If there is such a tuple, then replace $L_1$ by 
$$\Psi_1(L_1): c^2_1, c^2_2, c^2_3, c^2_4, c^2_5, c^2_6 \mid f^2_1, f^2_2, f^2_3, f^2_4 \mid \Psi_1(t^1_1), \Psi_1(t^1_2), $$
and go to Step 3, otherwise $L_1$ and $L_2$ are not link-homotopic. 
\item Let $\mathcal{S}$ be the set of the compositions of the relations $\psi_{ij}$ which do not change $f_i$ $(1\leq i \leq 4)$. Find a relation $\Psi_2\in\mathcal{S}$ which satisfies $\Psi_2\Psi_1(t^1_i)=t^2_i$ $(i=1,2)$. 
If there is such a relation, then $L_2=\Psi_2\Psi_1(L_1)$ and $L_1$ and $L_2$ are link-homotopic, otherwise they are not.
\end{enumerate}
\begin{remark}
From Table \ref{tablecase}, the commutators $[\psi_{ij}, \psi_{kl}]$ of $\psi_{ij}$ are in $\mathcal{S}$. 
Thus, in Step 2, we only need to consider the form $\psi_{14}^{a_8}\psi_{34}^{a_7}\psi_{23}^{a_6}\psi_{43}^{a_5}\psi_{32}^{a_4}\psi_{12}^{a_3}\psi_{41}^{a_2}
\psi_{21}^{a_1}$ up to $\mathcal{S}$. 
Moreover, in Step 3, to find generators of $\mathcal{S}$, we only need to consider the form $\psi_{14}^{b_8}\psi_{34}^{b_7}\psi_{23}^{b_6}\psi_{43}^{b_5}\psi_{32}^{b_4}\psi_{12}^{b_3}\psi_{41}^{b_2}
\psi_{21}^{b_1}$ other than commutators $[\psi_{ij}, \psi_{kl}]$ (see below). 

\end{remark}
\par
In Step 2, we can have the tuple $(a_1, \dots, a_8)$ by solving the system of linear Diophantine equations  $A{\bm a}={\bm f}$, where
$$A= \left(\begin{array}{cccccccc}
0 & 0 & 0 & c_6 & c_5 & -c_5 & -c_1 & 0 \\
0 & c_6 & 0 & 0 & -c_4 & 0 & c_2 & -c_2 \\
c_5 & -c_5 & -c_4 & 0 & 0 & 0 & 0 & c_3 \\
-c_1 & 0 & c_2 & -c_2 & 0 & c_3 & 0 & 0 \\
\end{array}
\right), 
{\bm a}=\left(\begin{array}{c}
a_1 \\
a_2 \\
\vdots \\
a_8 \\
\end{array}
\right), 
{\bm f}=\left(\begin{array}{c}
f^2_1-f^1_1 \\
f^2_2-f^1_2 \\
f^2_3-f^1_3 \\
f^2_4-f^1_4 \\
\end{array}
\right).
$$
\par
In Step 3, the commutators $[\psi_{ij}, \psi_{kl}]$ of $\psi_{ij}$ are in $\mathcal{S}$. 
The set $\mathcal{S}$ is generated by $[\psi_{ij}, \psi_{kl}]$ and the forms  $\Psi=\psi_{14}^{b_8}\psi_{34}^{b_7}\psi_{23}^{b_6}\psi_{43}^{b_5}\psi_{32}^{b_4}\psi_{12}^{b_3}\psi_{41}^{b_2}
\psi_{21}^{b_1}$ $(b_i\in \mathbb{Z})$ which satisfy $\Psi(f_i)=f_i$ $(1\leq i \leq 4)$. 
We can obtain the generators of $\mathcal{S}$ which have the form $\Psi=\psi_{14}^{b_8}\psi_{34}^{b_7}\psi_{23}^{b_6}\psi_{43}^{b_5}\psi_{32}^{b_4}\psi_{12}^{b_3}\psi_{41}^{b_2}
\psi_{21}^{b_1}$ by solving the system of linear Diophantine equations  $A{\bm b}={\bm 0}$, where 
$${\bm b}= \left(\begin{array}{c}
b_1 \\
b_2 \\
\vdots \\
b_8 \\
\end{array}
\right).$$
This system is the associated homogeneous system of $A{\bm a}={\bm f}$. Let $\overline{\Psi}_1, \dots , \overline{\Psi}_s$ be the generators of $\mathcal{S}$ corresponding to the generators of the solutions of $A{\bm b}={\bm 0}$. 
Note that $[\psi_{ij}, [\psi_{kl}, \psi_{mn}]]=0$
and $[[\psi_{i_1j_1}, \psi_{k_1l_1}], [\psi_{i_2j_2}, \psi_{k_2l_2}]]=0$ hold for the relations in Table \ref{originaltable}.
Then, a relation $\Psi_2$ in $\mathcal{S}$ can be represented as 
$$\Psi_2=\prod [\psi_{ij}, \psi_{kl}]^{d_{ijkl}} \cdot \overline{\Psi}^{d_s}_s \cdots \overline{\Psi}^{d_1}_1,$$
where the product runs over all commutators in Table \ref{tablecase} and $d_i, d_{ijkl} \in \mathbb{Z}$. 
Moreover, the form $\Psi_2$ with any integers $d_i$ and $d_{ijkl}$ is in $\mathcal{S}$. 
Finally, we can find $\Psi_2 \in \mathcal{S}$ which satisfies $\Psi_2\Psi_1(t^1_i)=t^2_i$ $(i=1,2)$ by solving the system of linear Diophantine equations corresponding to $\Psi_2\Psi_1(t^1_i)=t^2_i$.

\begin{example} \label{algorithm-example1}
We show the following two links are not link-homotopic by the algorithm:
$$L_1: 1,2,2,4,2,2\mid 0,0,0,0 \mid 0,0 \quad \mbox{ and }\quad L_2: 1,2,2,4,2,2\mid -2,0,2,1\mid 1,1.$$
\par
Remark that they are not in the subsets of $\mathcal{L}_4$ which are classified by invariants in Subsection \ref{invariants}. Moreover, they have the same Milnor homotopy invariants. Indeed, for the length 2, they have the same linking numbers ($-c_i$). For the length 3, $f_i$ modulo the greatest common divisor of $c_i$ in the same column in Table \ref{originaltable} is equal to the Milnor homotopy invariant with corresponding length 3 sequence $I$ or it with minus sign (see Remark \ref{intro-remark} and \ref{intro-remark2}) and $L_1$ and $L_2$ have the same value. For the length 4, since they have the linking number $-c_1=-1$ and $\Delta_{L_i}(I)=1$ for any length 4 sequence $I$, all Milnor homotopy invariants vanish. 
\par
We apply the algorithm to the links $L_1$ and $L_2$. 
The integers $c_i$ $(1\leq i \leq 6)$ of both links are equal. We solve $A{\bm a}={\bm f}$, where
$$A= \left(\begin{array}{cccccccc}
0 & 0 & 0 & 2 & 2 & -2 & -1 & 0 \\
0 & 2 & 0 & 0 & -4 & 0 & 2 & -2 \\
2 & -2 & -4 & 0 & 0 & 0 & 0 & 2 \\
-1 & 0 & 2 & -2 & 0 & 2 & 0 & 0 \\
\end{array}
\right), 
{\bm a}=\left(\begin{array}{c}
a_1 \\
a_2 \\
\vdots \\
a_8 \\
\end{array}
\right), 
{\bm f}=\left(\begin{array}{c}
-2 \\
0 \\
2 \\
1 \\
\end{array}
\right).
$$
Transform $A$ to the Smith normal form. If  we put 
$P=\left(\begin{array}{cccc}
-1 & 0 & 0 & 0 \\
0 & 0 & 0 & -1 \\
0 & 0 & -1 & -2 \\
2 & 1 & 1 & 2 \\
\end{array}
\right)$ and $Q=\left(\begin{array}{cccccccc}
0 & 1 & 0 & -2 & 0 & 2 & 2 & 0 \\
0 & 0 & 1 & -2 & 0 & 2 & 0 & 1 \\
0 & 0 & 0 & 0 & 0 & 0 & 1 & 0 \\
0 & 0 & 0 & 1 & 0 & 0 & 0 & 0 \\
0 & 0 & 0 & 0 & 1 & 0 & 0 & 0 \\
0 & 0 & 0 & 0 & 0 & 1 & 0 & 0 \\
1 & 0 & 0 & 2 & 2 & -2 & 0 & 0 \\
0 & 0 & 0 & 0 & 0 & 0 & 0 & 1 \\
\end{array}
\right)$, then $D=PAQ=\left(\begin{array}{cccccccc}
1 & 0 & 0 & 0 & 0 & 0 & 0 & 0 \\
0 & 1 & 0 & 0 & 0 & 0 & 0 & 0 \\
0 & 0 & 2 & 0 & 0 & 0 & 0 & 0 \\
0 & 0 & 0 & 0 & 0 & 0 & 0 & 0 \\
\end{array}
\right)$. 
Here $P \in GL_4(\mathbb{Z})$ and $Q \in GL_8(\mathbb{Z})$. 
\par
We have the system of linear Diophantine equations $D{\bm c}={\bm f}'$ which is equivalent to $A{\bm a}={\bm f}$, where ${\bm c}=Q^{-1}{\bm a}$ and ${\bm f}'=P{\bm f}=\left(\begin{array}{c}
2 \\
-1 \\
-4 \\
0 \\
\end{array}
\right)$. 
($A{\bm a}={\bm f} \Leftrightarrow PA{\bm a}=P{\bm f} \Leftrightarrow PAQQ^{-1}{\bm a}=P{\bm f}\Leftrightarrow D{\bm c}={\bm f}'$.)
The solutions are
$${\bm c}= \left(\begin{array}{c}
2 \\ -1 \\ -2 \\ 0 \\ 0 \\ 0 \\ 0 \\ 0 \\
\end{array}
\right)
+h_1
\left(\begin{array}{c}
0 \\ 0 \\ 0 \\ 1 \\ 0 \\ 0 \\ 0 \\ 0 \\
\end{array}
\right)
+h_2
\left(\begin{array}{c}
0 \\ 0 \\ 0 \\ 0 \\ 1 \\ 0 \\ 0 \\ 0 \\
\end{array}
\right)
+h_3
\left(\begin{array}{c}
0 \\ 0 \\ 0 \\ 0 \\ 0 \\ 1 \\ 0 \\ 0 \\
\end{array}
\right)
+h_4
\left(\begin{array}{c}
0 \\ 0 \\ 0 \\ 0 \\ 0 \\ 0 \\ 1 \\ 0 \\
\end{array}
\right)
+h_5
\left(\begin{array}{c}
0 \\ 0 \\ 0 \\ 0 \\ 0 \\ 0 \\ 0 \\ 1 \\
\end{array}
\right)
\quad(h_i \in \mathbb{Z}).$$ 
Then the solutions of $A{\bm a}={\bm f}$ are
\begin{equation*}
\begin{split}
&{\bm a}=Q{\bm c}\\
&=
\left(\begin{array}{c}
-1 \\ -2 \\ 0 \\ 0 \\ 0 \\ 0 \\ 2 \\ 0 \\
\end{array}
\right)
+h_1
\left(\begin{array}{c}
-2 \\ -2 \\ 0 \\ 1 \\ 0 \\ 0 \\ 2 \\ 0 \\
\end{array}
\right)
+h_2
\left(\begin{array}{c}
0 \\ 0 \\ 0 \\ 0 \\ 1 \\ 0 \\ 2 \\ 0 \\
\end{array}
\right)
+h_3
\left(\begin{array}{c}
2 \\ 2 \\ 0 \\ 0 \\ 0 \\ 1 \\ -2 \\ 0 \\
\end{array}
\right)
+h_4
\left(\begin{array}{c}
2 \\ 0 \\ 1 \\ 0 \\ 0 \\ 0 \\ 0 \\ 0 \\
\end{array}
\right)
+h_5
\left(\begin{array}{c}
0 \\ 1 \\ 0 \\ 0 \\ 0 \\ 0 \\ 0 \\ 1 \\
\end{array}
\right)
\,\,\,(h_i \in \mathbb{Z}).
\end{split}
\end{equation*}
Take $(a_1, \dots, a_8)=(1,0,0,0,0,1,0,0)$ $(h_3=1$ and $h_i=0$ $(i\neq 3))$, then $\Psi_1= \psi_{23}\psi_{21}$ and have
$$\Psi_1(L_1): 1,2,2,4,2,2\mid -2,0,2,1\mid 0,-2$$
The solutions of $A{\bm b}={\bm 0}$ are 
$${\bm b}=
h_1
\left(\begin{array}{c}
-2 \\ -2 \\ 0 \\ 1 \\ 0 \\ 0 \\ 2 \\ 0 \\
\end{array}
\right)
+h_2
\left(\begin{array}{c}
0 \\ 0 \\ 0 \\ 0 \\ 1 \\ 0 \\ 2 \\ 0 \\
\end{array}
\right)
+h_3
\left(\begin{array}{c}
2 \\ 2 \\ 0 \\ 0 \\ 0 \\ 1 \\ -2 \\ 0 \\
\end{array}
\right)
+h_4
\left(\begin{array}{c}
2 \\ 0 \\ 1 \\ 0 \\ 0 \\ 0 \\ 0 \\ 0 \\
\end{array}
\right)
+h_5
\left(\begin{array}{c}
0 \\ 1 \\ 0 \\ 0 \\ 0 \\ 0 \\ 0 \\ 1 \\
\end{array}
\right)
\quad(h_i \in \mathbb{Z}).$$
Then the generators of $\mathcal{S}$ in the form of  $ \psi_{14}^{b_8}\psi_{34}^{b_7}\psi_{23}^{b_6}\psi_{43}^{b_5}\psi_{32}^{b_4}\psi_{12}^{b_3}\psi_{41}^{b_2}
\psi_{21}^{b_1}$ are 
$$
\mbox{$\overline{\Psi}_1=\psi_{34}^{2}\psi_{32}\psi_{41}^{-2}
\psi_{21}^{-2}$, $\overline{\Psi}_2= \psi_{34}^{2}\psi_{43}$, 
$\overline{\Psi}_3= \psi_{34}^{-2}\psi_{23}\psi_{41}^{2}\psi_{21}^{2}$, 
$\overline{\Psi}_4= \psi_{12}\psi_{21}^2$ 
and 
$\overline{\Psi}_5= \psi_{14}\psi_{41}$.}$$

\begin{table}[htb] 
\begin{center}
   \caption{Changes of the numbers of the claspers by $\overline{\Psi}_i$.
   }\label{S-table}
  \begin{tabular}{|c|c|c|c|c|c|c|} \hline
      & $f_1$ & $f_2$ & $f_3$ & $f_4$ & $t_1$ & $t_2$ \\ \hline 
    $\overline{\Psi}_1$ & 0 & $0$ & $0$ & $0$ & $0$ & 4 \\ \hline 
    $\overline{\Psi}_2$ & 0 & $0$ & $0$ & $0$ & 0 & 0 \\ \hline 
    $\overline{\Psi}_3$ & 0 & $0$ & $0$ & $0$ & 0 & 0 \\ \hline 
    $\overline{\Psi}_4$ & 0 & $0$ & $0$ & $0$ & 0 & 0 \\ \hline     
    $\overline{\Psi}_5$ & 0 & $0$ & $0$ & $0$ & 0 & 0 \\ \hline 
  \end{tabular} \\
\end{center}
\end{table}

These $\overline{\Psi}_i$ and commutators $[\psi_{ij}, \psi_{kl}]$ in Table \ref{tablecase} generate $\mathcal{S}$. From Table \ref{tablecase} and Table \ref{S-table}, the combination of the generators can not change $t_1$ to $t_1+1$. Thus there is no relation in $\mathcal{S}$ which maps $\Psi_1(L_1)$ to $L_2$ and $L_1$ and $L_2$ are not link-homotopic.  

\end{example}

\begin{example} \label{algorithm-example2}
We show the following two links are link-homotopic by the algorithm:
$$L_1: 1,2,2,4,2,2\mid 0,0,0,0\mid 0,0\quad \mbox{and}\quad L_3: 1,2,2,4,2,2\mid -2,0,2,1\mid 2,0.$$
Similar to Example \ref{algorithm-example1}, put  $\Psi_1= \psi_{23}\psi_{21}$ and have
$$\Psi_1(L_1): 1,2,2,4,2,2\mid -2,0,2,1\mid 0,-2.$$
By putting $\Psi_2= [\psi_{43}, \psi_{14}][\psi_{14}, \psi_{21}]^{2} (\in \mathcal{S})$, then $\Psi_{2}\Psi_1(L_1)=L_3$. Thus $L_1$ and $L_3$ are link-homotopic.  
\end{example}



\end{document}